\documentclass[reqno]{amsart}
\usepackage{amsmath,amsfonts,amssymb,amsthm,amscd,yfonts, mathtools,mathrsfs}
\usepackage{hyperref}
\usepackage{mathtools}
\usepackage{tensor}
\usepackage[makeroom]{cancel}
\usepackage[margin=1.6in]{geometry}
\usepackage{enumitem}
\usepackage{tabu}
\usepackage{tikz, tikz-cd}

\newcommand{\arxiv}[1]{\href{http://arxiv.org/abs/#1}{\tt arXiv:\nolinkurl{#1}}}

\definecolor{darkblue}{HTML}{111199}
\definecolor{darkgreen}{HTML}{336633}
\definecolor{darkred}{HTML}{993333}
\definecolor{darkpurple}{HTML}{995599}

\makeatletter

\def\down{\vee}
\def\up{\wedge}
\def\inc{\operatorname{in}}
\def\pr{\operatorname{pr}}
\def\half{{\textstyle\frac{1}{2}}}
\def\threehalves{{\textstyle\frac{3}{2}}}

\def\isig{i\!\operatorname{-sig}}
\def\iSIG{i\!\operatorname{-}\mathbf{sig}}
\def\la{\lambda}
\def\eps{\varepsilon}
\def\0{{\bar{0}}}
\def\1{{\bar{1}}}
\def\ba{\text{\boldmath$a$}}
\def\bb{\text{\boldmath$b$}}
\def\bz{\text{\boldmath$z$}}
\def\bbeta{\text{\boldmath$\beta$}}
\def\bomega{\text{\boldmath$\omega$}}
\def\bgamma{\text{\boldmath$\gamma$}}
\def\bc{\text{\boldmath$c$}}

\def\bd{\text{\boldmath$d$}}
\def\bj{\text{\boldmath$j$}}
\def\bsigma{\text{\boldmath$\sigma$}}

\def\g{{\mathfrak{g}}}
\def\b{{\mathfrak{b}}}
\def\t{{\mathfrak{t}}}
\def\h{{\mathfrak{h}}}
\def\C{{\mathcal{C}}}
\def\sO{s\mathcal{O}}
\def\sF{s\mathcal{F}}

\def\O{\mathcal{O}}
\def\F{\mathcal{F}}

\def\ob{\operatorname{ob}}

\def\pr{{\operatorname{pr}}}
\def\sF{s\hspace{-0.26mm}F}
\def\sE{s\hspace{-0.26mm}E}

\def\Web{A}
\def\AHC{{\mathcal{AHC}}}

\def\QH{{\mathcal{QH}}}

\newtheorem{theorem}{Theorem}[section]
\newtheorem{lemma}[theorem]{Lemma}

\newtheorem{corollary}[theorem]{Corollary} 
  
\theoremstyle{definition}  
\newtheorem{definition}[theorem]{Definition}

\newtheorem{example}[theorem]{Example}

\newtheorem{remark}[theorem]{Remark}

\@addtoreset{equation}{section}

\DeclareSymbolFont{bbold}{U}{bbold}{m}{n}
\DeclareSymbolFontAlphabet{\mathbbold}{bbold}
\DeclareMathOperator\End{End}
\DeclareMathOperator\Hom{Hom}

\DeclareMathOperator\N{\mathbb{N}}
\DeclareMathOperator\Z{\mathbb{Z}}
\DeclareMathOperator\CC{\mathbb{C}}
\DeclareMathOperator\K{{\CC}}

\DeclareMathOperator\grmod{grmod}

\def\SVec{\mathcal{SV}ec}
\def\SMod{\mathcal{S\!M}od}
\def\Mod{\mathcal{M}od}
\def\mod{mod}
\def\grmod{grmod}

\DeclareMathOperator\D{\mathcal{D}}
\DeclareMathOperator\B{\mathbf{B}}

\DeclareMathOperator\WT{\mathbf{wt}}
\DeclareMathOperator\wt{\operatorname{wt}}

\DeclareMathOperator\Q{\mathbb{Q}}

\def\q{{\mathfrak{q}}}

\begin{document}

\title[Type C blocks]{\boldmath Type C blocks of super category $\O$}
\author{Jonathan Brundan and Nicholas Davidson}

\begin{abstract}
We show that the blocks of category $\O$
for the Lie superalgebra $\q_n(\K)$ associated to
half-integral weights
carry the structure of a tensor product categorification for the infinite rank Kac-Moody algebra of
type C$_\infty$.
This allows us to prove two conjectures formulated by Cheng, Kwon and Lam. 
%in [S.-J. Cheng, J.-H. Kwon, and W. Wang, Character formulae for queer Lie superalgebras and
%canonical bases of type C; {\tt arXiv:1512.00116v2}].
We then focus on the full subcategory consisting of finite-dimensional
representations, which we show is a highest weight category with
blocks that are Morita equivalent to certain generalized Khovanov
arc algebras.
\end{abstract}
\thanks{Research supported in part by NSF grant DMS-1161094.}

\address{Department of Mathematics, University of Oregon, Eugene, OR 97403, USA}
\email{brundan@uoregon.edu}

\address{Department of Mathematics, University of Oklahoma, Norman,
  OK 73019, USA}
\email{njd@ou.edu}

\maketitle

\section{Introduction}

In this article, we apply some powerful tools from higher
representation theory to the study of the
BGG category $\O$ for the
Lie superalgebra $\mathfrak{q}_n(\K)$,
and its subcategory $\F$ of finite-dimensional representations. We restrict our attention
throughout to modules with half-integral weights.
In fact, by \cite{C}, the study of the category $\O$ for $\mathfrak{q}_n(\K)$
reduces to studying three types of blocks, known as the type A, type B, and
type C blocks. The half-integral weight case studied here constitutes
all of the type C blocks. For types A and B blocks, we refer the
reader to \cite{CKW, BD} and \cite{B1, CKW, D2}, respectively.

The type C blocks are already known to be
highest weight categories in the sense of \cite{CPS}.
We will prove two conjectures about them
formulated by Cheng, Kwon and Wang, namely, \cite[Conjectures 5.12--5.13]{CKW}. Roughly speaking,
these assert that the combinatorics of type C blocks is controlled by certain canonical
bases
for the tensor power $V^{\otimes n}$ of the minuscule natural representation
$V$ of the quantum group of type C$_\infty$. Actually, in general, one
needs to consider Webster's
``orthodox basis'' from \cite{Web1}, which is subtley different from Lusztig's canonical basis. Since there
is no elementary algorithm to compute Webster's basis, this is still
not an entirely satisfactory picture.

Interest in the category $\F$ (again, for half-integral weights) 
was rekindled by another recent paper of Cheng and Kwon \cite{CK}.
We will show here that $\F$ is a highest weight
category, answering \cite[Question
5.1(1)]{CKW}.
When combined with the main result of \cite{BS4}, our approach 
actually allows us to describe $\F$ in purely diagrammatical
terms: its blocks are equivalent to 
finite-dimensional modules over the generalized Khovanov arc algebras
denoted $K^{+\infty}_r$ in \cite{BS1}. 

The remainder of the article is organized as follows.
\begin{itemize}
\item
In section 2, we set up the underlying combinatorics of the
$\mathfrak{sp}_{2\infty}$-module $V^{\otimes n}$. As observed already
in \cite{CKW}, this may be identified with the Grothendieck group of the
category $\O^\Delta$ of $\Delta$-filtered modules of the category $\O$
to be studied later in the paper. We also give a brief review of 
Lusztig's canonical basis for this module, including an elementary
algorithm to compute it in practice, and recall
\cite[Proposition 4.1]{CKW}, which relates this 
type C canonical basis to some other type A canonical bases.
\item
In section 3, we introduce the supercategory $\sO$ for the Lie
superalgebra $\q_n(\K)$
and all half-integral weights.
Actually, when $n$ is odd, it is more convenient
to work with supermodules over $\q_n(\K) \oplus \q_1(\K)$ following the idea of
\cite{BD}. This means that the supercategory $\sO$ considered here in the odd case 
is the Clifford twist of the one appearing in \cite{CKW}. This trick
unifies our treatment of the even and odd cases, and actually makes
our results slightly stronger for odd $n$.
Mimicking the approach of \cite{BD},
we then show that $\sO$ splits as $\O \oplus \Pi \O$ for a highest weight
category $\O$, and that $\O$ admits the structure of a tensor product
categorification of the $\mathfrak{sp}_{2\infty}$-module $V^{\otimes
  n}$ in the general sense of Losev and Webster \cite{LW}.
Our proof depends crucially on a particular instance of the remarkable isomorphisms
discovered by Kang, Kashiwara and Tsuchioka \cite{KKT}.
\item
In section 4, 
we combine our main
result from section 3 with an argument involving truncation from
$\mathfrak{sp}_{2\infty}$ to $\mathfrak{sp}_{2k}$ and the uniqueness of $\mathfrak{sp}_{2k}$-tensor product categorifications
established in \cite{LW}, in order to prove
the first Cheng-Kwon-Wang conjecture.
This is similar to the proof of the
Kazhdan-Lusztig conjecture for the general linear supergroup 
given in \cite{BLW}.
We also give an application to classifying
the indecomposable projective-injective supermodules in $\sO$.
\item
In section 5, we use another form of truncation, this time
from $\mathfrak{sp}_{2\infty}$ to $\mathfrak{sl}_{+\infty}$,
to establish the second Cheng-Kwon-Wang conjecture. In fact, we show that the category $\O$
admits a filtration whose sections are 
$\mathfrak{sl}_{+\infty}$-tensor product categorifications, a result which may be viewed as
a categorical version of \cite[Proposition 4.1]{CKW}.
When combined with the uniqueness of $\mathfrak{sl}_{+\infty}$-tensor product categorifications
established in \cite{BLW}, this also allows us to
understand the structure of the subcategory $\F$ of $\O$ consisting
of the finite-dimensional supermodules: we show that $\F$ 
decomposes as
$$
\F = \bigoplus_{n_0+n_1=n}
\F_{n_0|n_1}
$$
with $\F_{n_0|n_1}$ being equivalent to a
quotient of the category of rational representations of the general
linear supergroup $GL_{n_0|n_1}(\CC)$. From this, we deduce that $\F$ is
a highest weight category, and its blocks are Morita equivalent
to certain generalized
Khovanov arc algebras
like in \cite{BS4}.
\end{itemize}

\vspace{2mm}
\noindent{\em Acknowledgements.}
We thank 
Shunsuke Tsuchioka for allowing us to include his
counterexamples to positivity in
Example~\ref{tsuchioka}.

\section{Canonical basis}

We are going to be interested in categorifications of certain tensor products of minuscule
representations of various Kac-Moody algebras. 
In this section, we define these tensor products
and make some elementary combinatorial observations about them.
Most of this material also be found in equivalent form in \cite{CKW}, but our conventions
are somewhat different.

\subsection{Minuscule representations}\label{par}
We will need the 
(complex) Kac-Moody algebras
of the following types:
$$
\begin{array}{|l|l|l|}
\hline
\text{Type}&\text{Dynkin diagram}&\text{Simple roots}\\\hline
\mathfrak{sl}_{\infty}
&
{\begin{picture}(113, 15)%
\put(18,2){\circle{4}}%
\put(37,2){\circle{4}}%
\put(56,2){\circle{4}}%
\put(75, 2){\circle{4}}%
\put(94, 2){\circle{4}}%
\put(20, 2){\line(1, 0){15.5}}%
\put(39, 2){\line(1, 0){15.5}}%
\put(58, 2){\line(1, 0){15.5}}%
\put(77, 2){\line(1, 0){15.5}}%
\put(98, 2){\line(1, 0){1}}%
\put(101, 2){\line(1, 0){1}}%
\put(104, 2){\line(1, 0){1}}%
\put(107, 2){\line(1, 0){1}}%
\put(110, 2){\line(1, 0){1}}%
\put(4, 2){\line(1, 0){1}}%
\put(1, 2){\line(1, 0){1}}%
\put(7, 2){\line(1, 0){1}}%
\put(10, 2){\line(1, 0){1}}%
\put(13, 2){\line(1, 0){1}}%
\put(15, 8){\makebox(0, 0)[b]{$_{{-2}}$}}%
\put(34, 8){\makebox(0, 0)[b]{$_{{-1}}$}}%
\put(56, 8){\makebox(0, 0)[b]{$_{0}$}}%
\put(75, 8){\makebox(0, 0)[b]{$_{1}$}}%
\put(94, 8){\makebox(0, 0)[b]{$_{{2}}$}}%
\end{picture}}
&
\alpha_i = \eps_i-\eps_{i+1}\\
\mathfrak{sl}_{+\infty}
&
{\begin{picture}(83, 15)%
\put(6, 2){\circle{4}}%
\put(25, 2){\circle{4}}%
\put(44, 2){\circle{4}}%
\put(8, 2){\line(1, 0){15.5}}%
\put(27, 2){\line(1, 0){15.5}}%
\put(48, 2){\line(1, 0){1}}%
\put(51, 2){\line(1, 0){1}}%
\put(54, 2){\line(1, 0){1}}%
\put(57, 2){\line(1, 0){1}}%
\put(60, 2){\line(1, 0){1}}%
\put(6, 8){\makebox(0, 0)[b]{$_{1}$}}%
\put(25, 8){\makebox(0, 0)[b]{$_{{2}}$}}%
\put(44, 8){\makebox(0, 0)[b]{$_{{3}}$}}%
\end{picture}}
&
\alpha_i = \eps_{i}-\eps_{i+1}
\\
\hline
\mathfrak{sp}_{2\infty}
&
{\begin{picture}(83, 15)%
\put(6,2){\circle{4}}%
\put(25, 2){\circle{4}}%
\put(44, 2){\circle{4}}%
\put(63, 2){\circle{4}}%
\put(8, 1){\line(1, 0){15}}%
\put(8, 3){\line(1, 0){15}}%
\put(27, 2){\line(1, 0){15.5}}%
\put(46, 2){\line(1, 0){15.5}}%
\put(67, 2){\line(1, 0){1}}%
\put(70, 2){\line(1, 0){1}}%
\put(73, 2){\line(1, 0){1}}%
\put(76, 2){\line(1, 0){1}}%
\put(79, 2){\line(1, 0){1}}%
\put(15, -0.8){\makebox(0, 0)[b]{$>$}}%
\put(6, 8){\makebox(0, 0)[b]{$_{0}$}}%
\put(25, 8){\makebox(0, 0)[b]{$_{1}$}}%
\put(44, 8){\makebox(0, 0)[b]{$_{{2}}$}}%
\put(63, 8){\makebox(0, 0)[b]{$_{{3}}$}}%
\end{picture}}
&
\alpha_0 = -2\eps_0,\:
\alpha_i = \eps_{i-1}-\eps_i\:(i > 0)
\\
\mathfrak{sp}_{2k}
%\:(k \geq 1)
&{\begin{picture}(87, 15)%
\put(6,2){\circle{4}}%
\put(25, 2){\circle{4}}%
\put(44, 2){\circle{4}}%
\put(63, 2){\circle{4}}%
\put(82, 2){\circle{4}}%
\put(8, 1){\line(1, 0){15}}%
\put(8, 3){\line(1, 0){15}}%
\put(27, 2){\line(1, 0){15.5}}%
\put(65, 2){\line(1, 0){15.5}}%
\put(48, 2){\line(1, 0){1}}%
\put(51, 2){\line(1, 0){1}}%
\put(54, 2){\line(1, 0){1}}%
\put(57, 2){\line(1, 0){1}}%
\put(60, 2){\line(1, 0){1}}%
\put(15, -0.8){\makebox(0, 0)[b]{$>$}}%
\put(6, 8){\makebox(0, 0)[b]{$_{0}$}}%
\put(25, 8){\makebox(0, 0)[b]{$_{1}$}}%
\put(44, 8){\makebox(0, 0)[b]{$_{{2}}$}}%
\put(63, 7.5){\makebox(0, 0)[b]{$_{{k-2}}$}}%
\put(83, 7.5){\makebox(0, 0)[b]{$_{{k-1}}$}}%
\end{picture}}
&
\alpha_0 = -2\eps_0,\:
\alpha_i = \eps_{i-1}-\eps_i\:(0 < i < k)\\
\hline
\end{array}
$$
Suppose that $\mathfrak{s}$ is one of these Lie algebras.
Letting $I$ denote the set that indexes the vertices of
the underlying Dynkin diagram in the above table, $\mathfrak{s}$ is generated by
its {\em Chevalley generators}
$\{e_i, f_i\:|\:i \in I\}$ subject to the usual Serre relations.
Let $\mathfrak{t}$ be the Cartan subalgebra spanned by $\left\{h_i :=
[e_i,f_i]\:|\:i \in I\right\}$.
We also introduce the {\em weight lattice} $P := \bigoplus_{i \in I}
\Z \eps_i$, which we identify with an Abelian subgroup of $\t^*$ so that the {\em simple roots} $\{\alpha_i\:|\:i \in
I\}$ of $\mathfrak{s}$
are identified with the elements of $P$ indicated in the table. Note
then that
\begin{equation}
\langle h_i, \alpha_j \rangle = 2\frac{(\alpha_i, \alpha_j)}{(\alpha_i,\alpha_i)}
\end{equation}
where $(\cdot,\cdot)$ is the bilinear form on $P$ defined
from  $(\eps_i, \eps_j) = \delta_{i,j}$.
There is a corresponding {\em dominance order} $\unrhd$ on $P$
defined from $\lambda \unrhd \mu$ if and only if $\lambda-\mu$ is a
sum of simple roots.
(The notation $I, P, \dots$ just introduced is potentially ambiguous as it
depends on the particular choice of $\mathfrak{s}$, but this should always be
clear from the context.)
 
As is evident from the Dynkin diagrams,
there are natural 
inclusions
$$
\mathfrak{sp}_2 \hookrightarrow \mathfrak{sp}_4 \hookrightarrow \mathfrak{sp}_6 \hookrightarrow \cdots
\hookrightarrow \mathfrak{sp}_{2\infty}
\hookleftarrow \mathfrak{sl}_{+\infty} \hookrightarrow \mathfrak{sl}_\infty
$$
sending Chevalley generators to Chevalley generators.
These embeddings will play an important role in our applications.

We proceed to introduce various minuscule representations of these Lie
algebras.

For $\mathfrak{sl}_\infty$, we will consider both its natural module
 $V^+$ and the dual $V^-$. These have standard bases $\{v_j^+\:|\:j \in \Z\}$
and $\{v_j^-\:|\:j \in \Z\}$, respectively. 
The weight of the
 vector $v_j^{\pm}$ is
$\pm \eps_j$,
and
the Chevalley generators act by
\begin{align}\label{e}
f_i v^+_j &= \left\{\begin{array}{ll}
v^+_{j+1}&\text{if $j = i$}\\
0&\text{otherwise,}
\end{array}\right.
&e_i v^+_j &= \left\{\begin{array}{ll}
v^+_{j-1}&\text{if $j = 1+i$}\\
0&\text{otherwise},
\end{array}\right.\\
f_i v^-_j &= \left\{\begin{array}{ll}
v^-_{j-1}&\text{if $j = 1+i$}\\
0&\text{otherwise},
\end{array}\right.
&e_i v^-_j &= \left\{\begin{array}{ll}
v^-_{j+1}&\text{if $j = i$}\\
0&\text{otherwise}.
\end{array}\right.\label{f}
\end{align}
Similarly, we have the natural and dual natural modules
for $\mathfrak{sl}_{+\infty}$, which will be denoted $V_0^+$ and
$V_0^-$, respectively.
Exploiting the inclusion
$\mathfrak{sl}_{+\infty} \hookrightarrow \mathfrak{sl}_\infty$,
we identify $V_0^\pm$ with the submodule of the restriction of $V^\pm$ spanned by $\{v_j^\pm\:|\:j > 0\}$.

For $\mathfrak{sp}_{2\infty}$, we only have its
natural module $V$. This has basis
$\left\{v_j\:|\:j \in \Z\right\}$,
with $v_j$ 
of weight $\eps_{j-1}$ if $j > 0$ or $-\eps_{-j}$ if $j \leq 0$,
and action defined from
\begin{align}\label{fe}
f_i v_j &= \left\{\begin{array}{ll}
v_{j+1}&\text{if $j = \pm i$}\\
0&\text{otherwise},
\end{array}\right.
&e_i v_j &= \left\{\begin{array}{ll}
v_{j-1}&\text{if $j = 1 \pm i$}\\
0&\text{otherwise}.
\end{array}\right.
\end{align}
Similarly, for any $k \geq 1$, we have the natural module
$V_k$ of $\mathfrak{sp}_{2k}$, which is identified with the
submodule of the restriction of $V$ spanned by $\{v_j\:|\:-k < j \leq k\}$.

\begin{lemma}\label{tummy}
As an $\mathfrak{sl}_{+\infty}$-module,
$V$
is isomorphic to $V_0^+
\oplus V_0^-$.
\end{lemma}

\begin{proof}
The map $v_j^+ \mapsto v_j$ defines an isomorphism between $V_0^+$ and
the $\mathfrak{sl}_{+\infty}$-submodule of $V$ spanned by
$\{v_j\:|\:j > 0\}$.
Similarly, the map $v_j^- \mapsto v_{1-j}$
defines an isomorphism between 
$V_0^-$ and the submodule spanned by $\{v_j\:|\:j \leq
0\}$.
\end{proof}

\subsection{Tensor products}\label{tp}
We are really interested in tensor powers of the minuscule
representations defined so far.
To introduce these, fix $n \geq 1$
and let $\B$ denote the set $\Z^n$ 
of $n$-tuples $\bb = (b_1,\dots,b_n)$ of integers.
Let $\bd_r \in \B$ be the tuple with 1 in its $r$th entry and $0$ in
all other places.
Also, for $k \geq 1$
and a tuple of signs $\bsigma = (\sigma_1,\dots,\sigma_n) \in \{\pm\}^n$, let
\begin{align}
\B_0 &:= \{\bb \in \B\:|\:b_1,\dots,b_n > 0\},\label{B0}\\
\B_k &:= \{\bb \in \B\:|\:-k < b_1,\dots,b_n \leq k\},\label{Bk}\\
\B_\bsigma &:=
\{\bb \in \B\:|\:b_r > 0\text{ if }\sigma_r = +,
b_r \leq 0\text{ if }\sigma_r = -\}.\label{Bsigma}
\end{align}

Let $V^{\otimes \bsigma}$ 
denote the $\mathfrak{sl}_\infty$-module
$V^{\sigma_1}\otimes\cdots\otimes V^{\sigma_n}$.
It has the natural monomial basis
$\left\{v_\bb^\bsigma := v_{b_1}^{\sigma_1}\otimes\cdots\otimes
v_{b_n}^{\sigma_n}\:\big|\:\bb \in \B\right\}$.
The action of the Chevalley generators of $\mathfrak{sl}_\infty$ on this basis is given
explicitly by 
\begin{align}\label{haha}
f_i v_{\bb}^\bsigma &= \sum_{\substack{1 \leq t \leq n \\ \isig_t^\bsigma(\bb) =\mathtt{f}}} v_{\bb+\sigma_t\bd_t}^\bsigma,
&
e_i v_{\bb}^\bsigma &= \sum_{\substack{1 \leq t \leq n\\ \isig_t^\bsigma(\bb) = \mathtt{e}}} v_{\bb-\sigma_r\bd_t}^\bsigma,
\end{align}
where
$\iSIG^\bsigma(\bb) = 
(\isig^\bsigma_1(\bb),\dots,\isig^\bsigma_n(\bb))$
is the {\em $i$-signature} of $\bb \in \B$ (with respect to
$\bsigma$) defined from
\begin{equation}\label{sigdef}
\isig^\bsigma_t(\bb) := 
\left\{
\begin{array}{ll}
\mathtt{f}&\text{if $(b_t,\sigma_t) = (i,+)$ or $(b_t,\sigma_t) = (1+i,-)$},\\
\mathtt{e}&\text{if $(b_t,\sigma_t) = (1+i,+)$ or $(b_t,\sigma_t) = (i,-)$},\\
\bullet&\text{otherwise.}
\end{array}
\right.
\end{equation}
Similarly, we have the $\mathfrak{sl}_{+\infty}$-module
$V_0^{\otimes\bsigma} = V_0^{\sigma_1}\otimes\cdots\otimes
V_0^{\sigma_n}$, which we identify with the submodule of $V^{\otimes
  \bsigma}$ spanned by $\{v_\bb^\bsigma\:|\:\bb \in \B_0\}$.
The projection
\begin{equation}\label{pr0}
\pr_0:V^{\otimes\bsigma} \twoheadrightarrow
V_0^{\otimes \bsigma},
\qquad
v_\bb^\bsigma \mapsto \left\{
\begin{array}{ll}
v_\bb^\bsigma&\text{if $\bb \in \B_0$},\\
0&\text{otherwise}
\end{array}
\right.
\end{equation}
is an $\mathfrak{sl}_{+\infty}$-module homomorphism.

We also have the $\mathfrak{sp}_{2\infty}$-module $V^{\otimes n}$,
with 
basis $\left\{v_\bb := v_{b_1}\otimes\cdots\otimes
  v_{b_n}\:|\:\bb \in \B\right\}$.
The action is given explicitly by the formulae
\begin{align}\label{cg}
f_i v_{\bb} &= \sum_{\substack{1 \leq t \leq n \\ \isig_t(\bb) = \mathtt{f}}} v_{\bb+\bd_t},
&
e_i v_{\bb} &= \sum_{\substack{1 \leq t \leq n\\ \isig_t(\bb) = \mathtt{e}}} v_{\bb-\bd_t},
\end{align}
where this time $\iSIG(\bb) = 
(\isig_1(\bb),\dots,\isig_n(\bb))$
is defined from
\begin{equation}\label{newisig}
\isig_t(\bb) := 
\left\{
\begin{array}{ll}
\mathtt{f}&\text{if }
b_t = \pm i,\\
\mathtt{e}&\text{if }
b_t = 1 \pm i,\\
 \bullet&\text{otherwise.}
\end{array}
\right.
\end{equation}
Similarly, we have the $\mathfrak{sp}_{2k}$-module $V_k^{\otimes n}$, which
is identified with the submodule of $V^{\otimes n}$
spanned by $\{v_\bb\:|\:\bb \in \B_k\}$.
The projection
\begin{equation}\label{prk}
\pr_k:V^{\otimes n} \twoheadrightarrow
V_k^{\otimes n},
\qquad
v_\bb \mapsto \left\{
\begin{array}{ll}
v_\bb &\text{if $\bb \in \B_k$},\\
0&\text{otherwise}
\end{array}
\right.
\end{equation}
is an $\mathfrak{sp}_{2k}$-module homomorphism.

From Lemma~\ref{tummy}, we see that the restriction of $V^{\otimes n}$
to the subalgebra 
$\mathfrak{sl}_{+\infty}$
is isomorphic to $\!\!\displaystyle\bigoplus_{\bsigma \in \{\pm\}^n}
\!\!\!  V_0^{\otimes\bsigma}$.
To write down an explicit isomorphism, 
introduce the function
\begin{equation}\label{primemap}
\B \rightarrow \B_0, \qquad
\bb \mapsto \bb'
\end{equation}
where $\bb'$ is the tuple with $r$th entry $b_r$ if $b_r > 0$ or
$1-b_r$ if $b_r \leq 0$.
This restricts to bijections $\B_\bsigma\stackrel{\sim}{\rightarrow}\B_0$ for
each $\bsigma \in \{\pm\}^n$.
Define a linear map
\begin{equation}\label{prsigma}
\pr_{\bsigma}:V^{\otimes n} \twoheadrightarrow V_0^{\otimes\bsigma},
\qquad
v_\bb \mapsto \left\{
\begin{array}{ll}
v_{\bb'}^\bsigma&\text{if $\bb \in \B_\bsigma$,}\\
0&\text{otherwise.}
\end{array}\right.
\end{equation}
Then:

\begin{lemma}\label{restrict}
The map
$$
\!\!\displaystyle\sum_{\bsigma\in\{\pm\}^n} \!\!\!\pr_\bsigma:
V^{\otimes n}
\stackrel{\sim}{\rightarrow} \!\!\bigoplus_{\bsigma\in\{\pm\}^n}
\!\!\!V_0^{\otimes\bsigma}
$$
is an isomorphism of $\mathfrak{sl}_{+\infty}$-modules.
\end{lemma}

\subsection{Bruhat order}\label{bos}
Next, we introduce some partial orders on the index set
$\B$. 
These orders arise in Lusztig's construction of
canonical bases for the spaces $V^{\otimes \bsigma}$ and $V^{\otimes
  n}$, which we'll review in more detail in the next subsection.
To define them, we need the {\em inverse dominance order} $\preceq$ on
$P^n$ from \cite[Definition 3.2]{LW}.
For $\bbeta = (\beta_1,\dots,\beta_n) \in P^n$, we write $|\bbeta|$
for $\beta_1+\cdots+\beta_n \in P$.
Then, $\preceq$ is defined by declaring that $\bbeta \preceq \bgamma$
if and only if $|\bbeta|=|\bgamma|$ and
$\beta_1+\cdots+\beta_s \unrhd \gamma_1+\cdots+\gamma_s$
for each $s=1,\dots,n$.
(Obviously, $\preceq$ depends on the particular Lie algebra $\mathfrak{s}$ being considered.)

We start with $\mathfrak{sl}_\infty$.
So fix $\bsigma \in \{\pm\}^n$.
Recall that 
the weight spaces of $V^{\pm}$ are one-dimensional with
$v_j^{\pm}$ of weight $\pm \eps_j$. There is
an injective map
\begin{equation*}
\WT^\bsigma:\B \hookrightarrow P^n,\qquad
\bb \mapsto (\wt^\bsigma_1(\bb),\dots,\wt^\bsigma_n(\bb))
\end{equation*}
with
$\wt^\bsigma_r(\bb) := \sigma_r \eps_{b_r}$; in particular,
$v_\bb^\bsigma$ is of weight $|\WT^\bsigma(\bb)|$.
The {\em $\mathfrak{sl}_\infty$-Bruhat order} $\preceq_\bsigma$ on $\B$ is defined
from 
\begin{equation}\label{o600}
\ba \preceq_\bsigma \bb \Leftrightarrow \WT^\bsigma(\ba) \preceq
\WT^\bsigma(\bb)
\end{equation}
in the inverse dominance order for $\mathfrak{sl}_{\infty}$.
The induced order on the subset $\B_0$ from
(\ref{B0}) is the {\em $\mathfrak{sl}_{+\infty}$-Bruhat order} $\preceq_{\bsigma}$.
Sometimes the following equivalent description of $\preceq_\bsigma$ 
from \cite[Lemma 4.2]{BD} is useful:

\begin{lemma}\label{slequiv}
For $i \in I$ (which is either $\Z$ or $\Z_+$ depending on whether we
are considering $\mathfrak{sl}_\infty$ or $\mathfrak{sl}_{+\infty}$)
and $1 \leq s \leq n$, we let
\begin{equation}\label{N1}
N_{[1,s]}^\bsigma(\bb,i)
:=
\#\{1 \leq r \leq s\:|\:b_r > i, \sigma_r = +\}
-
\#\{1 \leq r \leq s\:|\:b_r > i, \sigma_r = -\}.
\end{equation}
Then, we have that $\ba \preceq_\bsigma \bb$ if and only if
\begin{itemize}
\item
$N_{[1,s]}^{\bsigma}(\ba,i) \leq N_{[1,s]}^\bsigma(\bb,i)$
for all $i \in I$ and $s = 1,\dots,n-1$;
\item
$N_{[1,n]}^\bsigma(\ba,i) = N_{[1,n]}^\bsigma(\bb,i)$
for all $i \in I$.
\end{itemize}
\end{lemma}

Turning our attention to $\mathfrak{sp}_{2\infty}$, we consider
instead the inclusion
\begin{equation*}
\WT:\B \hookrightarrow P^n,\qquad
\bb \mapsto (\wt_1(\bb),\dots,\wt_n(\bb))
\end{equation*}
defined by setting $\wt_r(\bb) := \eps_{b_r-1}$ if $b_r > 0$ or
$-\eps_{-b_r}$ if $b_r \leq 0$; in particular $v_\bb$ is of weight $|\WT(\bb)|$.
Then we define the {\em $\mathfrak{sp}_{2\infty}$-Bruhat order} $\preceq$ on $\B$ 
as before: 
\begin{equation}\label{o700}
\ba \preceq\bb\Leftrightarrow \WT(\ba) \preceq \WT(\bb)
\end{equation}
in the inverse dominance order for $\mathfrak{sp}_{2\infty}$.
The {\em $\mathfrak{sp}_{2k}$-Bruhat order} $\preceq$ 
is the induced order on the subset $\B_k$ from
(\ref{Bk}). 
There is a similar characterization of these orders to
Lemma~\ref{slequiv}:

\begin{lemma}\label{spequiv}
For $i \in I$ (which is either $\N$ or $\{0,1,\dots,k-1\}$ for
$\mathfrak{sp}_{2\infty}$ or $\mathfrak{sp}_{2k}$)
and $1 \leq s \leq n$, we let
\begin{equation}\label{N2}
N_{[1,s]}(\bb,i)
:=
\#\{1 \leq r \leq s\:|\:b_r > i\}
-
\#\{1 \leq r \leq s\:|\:b_r \leq -i\}.
\end{equation}
Then, we have that 
$\ba \preceq \bb$ if and only if
\begin{itemize}
\item $N_{[1,s]}(\ba,0) \equiv N_{[1,s]}(\bb,0)\pmod{2}$
for each $s=1,\dots,n-1$;
\item
$N_{[1,s]}(\ba, i) \leq N_{[1,s]}(\bb, i)$
for all $i \in I$ and $s = 1,\dots,n-1$;
\item
$N_{[1,n]}(\ba, i) = N_{[1,n]}(\bb, i)$
for all $i \in I$.
\end{itemize}
\end{lemma}

Recall the set $\B_\bsigma$ from (\ref{Bsigma})
and the bijection $\B_\bsigma \stackrel{\sim}{\rightarrow}\B_0, \bb
\mapsto \bb'$
from (\ref{primemap}).

\begin{lemma}\label{threw}
The map $\bb \mapsto \bb'$ defines a
poset isomorphism
 $(\B_\bsigma, \preceq) \stackrel{\sim}{\rightarrow}
(\B_0, \preceq_\bsigma)$.
\end{lemma}

\begin{proof}
This follows easily from the characterizations of the two Bruhat orders
that we have given, on noting from (\ref{N1})--(\ref{N2}) 
that 
$N_{[1,s]}(\bb,0) = \sigma_1 1 +\cdots + \sigma_s
1$ and
$N_{[1,s]}(\bb,i) = N_{[1,s]}^\bsigma(\bb',i)$
for $\bb \in \B_\bsigma$
and all $i > 0$.
\end{proof}

The remaining lemmas in this subsection are concerned with the
case $\mathfrak{s}=\mathfrak{sp}_{2\infty}$.

\begin{lemma}\label{not}
Suppose that $\ba \succeq \bb$ and 
$\isig_r(\ba) = \isig_n(\bb) =
\mathtt{f}$ 
for some $i \in I$ and
$1 \leq r \leq n$.
Then $\ba + \bd_r \succeq \bb + \bd_n$,
with equality if and only if $\ba = \bb$ and $r=n$.
\end{lemma}

\begin{proof}
This may be checked directly
from the characterization of the Bruhat order given by Lemma~\ref{spequiv}.
\end{proof}

\begin{lemma}\label{construction2}
For $\bb \in \B$,
there exists $\ba \in \B$ and a monomial $X$ in the
Chevalley generators
$\{f_i\:|\:i \in I\}$ of $\mathfrak{sp}_{2\infty}$
such that
\begin{itemize}
\item
$a_1 > \cdots > a_n$ and 
$a_r+a_s \neq 1$ for all $1 \leq r < s \leq n$;
\item
$
X v_\ba = v_\bb + 
(\text{a sum of $v_\bc$'s
for $\bc \succ \bb$})$.
\end{itemize}
\end{lemma}

\begin{proof}
We first explain an explicit construction for $\ba$ and $X$.
Suppose we are given $\bb \in \B$.
Define $\ba \in \B$ by setting $a_1 := b_1$, then inductively defining 
each $a_s$ 
for $s=2,\dots,n$ 
to be the greatest integer such that $a_s \leq b_s$ and 
$a_s \leq \min(a_r-1,-b_r)$ for all $1
\leq r < s$.
It is clear from the definition of $\ba$ that $a_1 >
\cdots > a_n$.
Also for $1 \leq r < s \leq n$, we have that 
$a_r+a_s \leq b_r - b_r = 0$.
 Then take
$X = X_n \cdots X_2$ 
where $
X_s := f_{|b_s-1|} \cdots f_{|a_s+1|} f_{|a_s|}.
$

To show that $X v_\ba = v_\bb + \text{(a sum of higher $v_\bc$'s)}$,
we proceed by induction on $n$, the result being trivial in case $n=1$.
For $n > 1$, let  
$\bar\ba := (a_1,\dots,a_{n-1})$, $\bar\bb
:= (b_1,\dots,b_{n-1})$
and 
$\bar X := X_{n-1} \cdots X_2$.
Applying the induction hypothesis in the $\mathfrak{sp}_{2\infty}$-module
$V^{\otimes(n-1)}$, we get that
$\bar X v_{\bar\ba} = v_{\bar\bb} + 
(\text{a sum of $v_{\bar\bc}$'s
for $\bar\bc \succ \bar\bb$}).
$
Now we observe that if $f_i$ is a Chevalley generator appearing in one of the monomials $X_r$ for
$r < n$
then $i \neq \pm a_n$, hence, $f_i v_{a_n} = 0$.
Letting
 $\tilde \bb := (b_1,\dots,b_{n-1}, a_n)$, we deduce that
$\bar X v_{\ba} = v_{\tilde\bb} + 
(\text{a sum of $v_{\tilde\bc}$'s
for $\tilde\bc \succ \tilde\bb$}).$
Finally we act with $X_n$, which sends $v_{a_n}$ to $v_{b_n}$, and
apply Lemma~\ref{not}.
\end{proof}

\subsection{Canonical basis}\label{notabasis}
So far, we have introduced the following tensor product modules
over various Lie algebras $\mathfrak{s}$:
$$
\begin{array}{|l|c|r|r|}
\hline
\mathfrak{s}&\text{Tensor space}&\text{Monomial basis}&\text{Canonical
  basis}\\\hline
\mathfrak{sl}_{\infty}&V^{\otimes\bsigma}&v_\bb^\bsigma\text{ for }\bb \in
\B_{\phantom{0}}&c_\bb^\bsigma\text{ for }\bb \in \B_{\phantom{0}}\\
\mathfrak{sl}_{+\infty}&V_0^{\otimes\bsigma}&v_\bb^\bsigma\text{ for }\bb \in
\B_0&\pr_0 c_\bb^\bsigma\text{ for }\bb \in \B_0\\\hline
\mathfrak{sp}_{2\infty}&V^{\otimes n}&v_\bb\text{ for }\bb \in
\B_{\phantom{k}}&c_\bb\text{ for }\bb \in \B_{\phantom{k}}\\
\mathfrak{sp}_{2k}&V_k^{\otimes n}&v_\bb\text{ for }\bb \in
\B_k&\pr_k c_\bb\text{ for }\bb \in \B_k\\
\hline
\end{array}
$$
In this subsection, we give meaning to the rightmost column of this
table by introducing some {\em
  canonical bases}, basically following a construction
of Lusztig from \cite[$\S$27.3]{Lubook}.

In each of the above cases, let $U_q \mathfrak{s}$ be the {\em quantized enveloping algebra}
associated to $\mathfrak{s}$ over the field $\Q(q)$ ($q$ an
indeterminate).
We denote the standard generators of $U_q \mathfrak{s}$ 
by $\{\dot e_i, \dot f_i, \dot k_i^{\pm}\:|\:i \in I\}$.
They are subject to the usual $q$-deformed Serre relations.
We view $U_q \mathfrak{s}$ as a Hopf algebra with comultiplication $\Delta$
defined from
\begin{equation*}
\Delta(\dot f_i) = 1 \otimes \dot f_i + \dot f_i \otimes \dot k_i,
\quad
\Delta(\dot e_i) = \dot k_i^{-1}\otimes\dot e_i + \dot e_i \otimes 1,
\quad
\Delta(\dot k_i) = \dot k_i \otimes \dot k_i.
\end{equation*}

The various minuscule representations introduced in $\S$\ref{par} all
have $q$-analogs; cf. \cite[$\S$5A.1]{Jantzen}. We will denote them by
decorating our earlier notation with a dot, so we have the
$\Q(q)$-vector spaces
$\dot V^{\pm}, \dot V^{\pm}_0, \dot V$ and $\dot V_k$ with bases
$\{\dot v_j^{\pm}\:|\:j \in \Z\}$,
$\{\dot v_j^{\pm}\:|\:j > 0\}$,
$\{\dot v_j\:|\:j \in \Z\}$ and
$\{\dot v_j\:|\:-k<j \leq k\}$, respectively.
The Chevalley generators $\dot f_i$ and $\dot e_i$ act on these bases
by the same formulae (\ref{e})--(\ref{fe}) as before, while the
diagonal action is given explicitly by
\begin{align*}
\dot k_i \dot v_j^+ &= q^{\delta_{i,j}-\delta_{1+i,j}} \dot v_j^+,
&\dot k_i \dot v_j^- &= q^{\delta_{1+i,j}-\delta_{i,j}} \dot
v_j^-,\end{align*}
for the $\mathfrak{sl}$ cases, or
$$
\dot k_i \dot v_j = q^{\delta_{i,j}+\delta_{-i,j}  - \delta_{1+i,j}
  -\delta_{1-i,j}} \dot v_j
$$
for $\mathfrak{sp}$.
Taking tensor products, we obtain the modules
$\dot V^{\otimes\bsigma}, \dot V^{\otimes\bsigma}_0, \dot V^{\otimes
  n}$ and $\dot V_k^{\otimes n}$, with their natural
monomial bases denoted now by
$\{\dot v_\bb^\bsigma\:|\:\bb \in \B\},
\{\dot v_\bb^\bsigma\:|\:\bb \in \B_0\},
\{\dot v_\bb\:|\:\bb \in \B\}$
and
$\{\dot v_\bb\:|\:\bb \in \B_k\}$, respectively.

In the infinite rank cases, we need to pass from the
$q$-tensor spaces just defined to completions in which
certain infinite sums of
the basis vectors
also make sense, as follows.

For $\mathfrak{sl}_\infty$, the completed tensor space is denoted
$\widehat{V}^{\otimes \bsigma}$. 
It is
the $\Q(q)$-vector space consisting of formal linear combinations
of the form $\sum_{\bb \in \B} p_\bb(q) \dot v^\bsigma_\bb$ for
rational functions $p_\bb(q) \in \Q(q)$
such that the {\em support} $\{\bb\in\B\:|\:p_\bb(q) \neq 0\}$
is contained in a finite union of sets of the form $\{\bb \in
\B\:|\:\WT^\bsigma(\bb) \succeq \bbeta\}$ for $\bbeta \in P^n$
(working with the inverse dominance order for $\mathfrak{sl}_\infty$).
This definition is justified in
\cite[Lemma 8.1]{BD}.
For $\mathfrak{sl}_{+\infty}$, exactly the
same procedure gives a completion $\widehat{V}^{\otimes \bsigma}_0$
of $\dot V^{\otimes\bsigma}_0$, which embeds naturally into
$\widehat{V}^{\otimes \bsigma}$. Also, as in (\ref{pr0}),
there is a projection
\begin{equation*}
\pr_0:
\widehat{V}^{\otimes\bsigma}
\twoheadrightarrow \widehat{V}^{\otimes\bsigma}_0,
\qquad
\sum_{\bb \in \B} p_\bb(q) \dot v_\bb^\bsigma \mapsto
\sum_{\bb \in \B_0} p_\bb(q) \dot v_\bb^\bsigma,
\end{equation*}
which is left inverse to the inclusion
$\inc_0:\widehat{V}^{\otimes\bsigma}_0 \hookrightarrow \widehat{V}^{\otimes\bsigma}$.

For $\mathfrak{sp}_{2\infty}$, we define
the completion $\widehat{V}^{\otimes
  n}$ of $\dot V^{\otimes n}$ in an analogous way, replacing the
$\mathfrak{sl}_\infty$-Bruhat order by the
$\mathfrak{sp}_{2\infty}$-Bruhat order.
So it is the $\Q(q)$-vector space consisting of formal linear combinations
of the form $\sum_{\bb \in \B} p_\bb(q) \dot v_\bb$
whose support
is contained in a finite union of sets of the form $\{\bb \in
\B\:|\:\WT(\bb) \succeq \bbeta\}$ for $\bbeta \in P^n$ (working with
the inverse dominance order for $\mathfrak{sp}_{2\infty}$).
Just like in \cite[Lemma 8.1]{BD}, the action of $U_q
\mathfrak{sp}_{2\infty}$
on $\dot V^{\otimes n}$ extends to an action on $\widehat{V}^{\otimes
  n}$, and the completion still splits as the direct sum of its weight spaces.
The $U_q \mathfrak{sp}_{2k}$-module $\dot V^{\otimes n}_k$ embeds naturally into 
$\dot V^{\otimes n}$, hence,  its completion $\widehat{V}^{\otimes
  n}$. As in (\ref{prk}), we also have the projection
\begin{equation*}
\pr_k:
\widehat{V}^{\otimes n}
\twoheadrightarrow \dot{V}^{\otimes n}_k,
\qquad
\sum_{\bb \in \B} p_\bb(q) \dot v_\bb \mapsto
\sum_{\bb \in \B_k} p_\bb(q) \dot v_\bb,
\end{equation*}
which is left inverse to the inclusion
$\inc_k:\dot{V}^{\otimes n}_k \hookrightarrow \widehat{V}^{\otimes n}$.

The projection (\ref{prsigma}) carries over to the present setting too:
there is a
$U_q\mathfrak{sl}_{+\infty}$-module homomorphism
\begin{equation*}
\pr_\bsigma:
\widehat{V}^{\otimes n}
\twoheadrightarrow \widehat{V}^{\otimes \bsigma}_0,
\qquad
\sum_{\bb \in \B} p_\bb(q) \dot v_\bb \mapsto
\sum_{\bb \in \B_\bsigma} p_\bb(q) \dot v_{\bb'}^{\bsigma}
\end{equation*}
for $\bsigma \in \{\pm\}^n$.
It is left inverse to
$\inc_\bsigma:\widehat{V}_0^{\otimes\bsigma} \hookrightarrow \widehat{V}^{\otimes
  n},
\sum_{\bb \in \B_\bsigma} p_\bb(q) \dot v_{\bb'}^{\bsigma} \mapsto
\sum_{\bb \in \B_\bsigma} p_\bb(q) \dot v_\bb$.

The key point now is that there are canonical bar involutions
on each of the spaces
$\widehat{V}^{\otimes\bsigma}, \widehat{V}_0^{\otimes\bsigma},
\widehat{V}^{\otimes n}$ and $\dot V^{\otimes n}_k$, which we'll
denote by $\psi, \psi_0, \psi$ and $\psi_k$, respectively.
Each one is antilinear with respect to the field automorphism 
$\Q(q)\rightarrow \Q(q), q \mapsto q^{-1}$, it preserves weight
spaces, and it commutes with all $\dot f_i$ and $\dot e_i$.
The construction in finite rank is explained in 
\cite[$\S$27.3.1]{Lubook} using the quasi-$R$-matrix $\Theta$; note
for this due to our different choice of $\Delta$ compared to \cite{Lubook}
that Lusztig's $v$ is our $q^{-1}$. The approach in infinite rank is essentially the same; one
needs 
the completion so that the infinite sums that arise
still make sense.
In the next paragraph, we go through the details of the definition of
$\psi:\widehat{V}^{\otimes n}\rightarrow \widehat{V}^{\otimes n}$
in the case of
$\mathfrak{sp}_{2\infty}$. The constructions for
$\mathfrak{sl}_{\infty}$
and $\mathfrak{sl}_{+\infty}$ are entirely analogous; see also
\cite[Lemma 8.2]{BD}.

So consider $\widehat{V}^{\otimes n}$.
Proceeding by induction on $n$, we set $\psi(\dot v_j) = \dot v_j$ for
each $j \in \Z$.
For $n > 1$, 
we assume that the analog $\bar\psi$ of $\psi$ on the space
$\widehat{V}^{\otimes(n-1)}$
has already been defined by induction.
Letting
$\bar\bb$ denote the $(n-1)$-tuple
$(b_1,\dots,b_{n-1})$, 
we define $\psi$ on $\widehat{V}^{\otimes n}$ by setting
\begin{equation}\label{bardef}
\psi\left({\sum_{\bb \in \B}} p_\bb(q) \dot v_\bb\right) := 
{\sum_{\bb \in \B}} p_\bb(q^{-1})
\,\Theta\!\left(\bar\psi(\dot v_{\bar\bb}\right) \otimes \dot
v_{b_n}).
\end{equation}
To better understand this expression, recall that the
quasi-$R$-matrix $\Theta$
is a formal sum of terms
$\Theta_\beta$
for $\beta \in \bigoplus_{i \in I} \N \alpha_i$, with $\Theta_0 = 1$ and
$\Theta_\beta \in (U^-_q \mathfrak{sp}_{2\infty})_{-\beta}\otimes (U_q^+
\mathfrak{sp}_{2\infty})_\beta$.
The only monomials in the generators of $U_q^+\mathfrak{sp}_{2\infty}$ that are non-zero
on $\dot v_j$
are of the form $\dot e_{|i|} \dot e_{|i+1|}\cdots \dot e_{|j-1|}$
for integers $i \leq j$. Hence, for any $v \in
\widehat{V}^{\otimes(n-1)}$ and $j \in \Z$, we have that
\begin{equation}\label{rhs}
\Theta\left(v\otimes \dot
v_j\right)
= v\otimes \dot
v_j+
\sum_{i < j} \left(\Theta_{i,j} v \right) \otimes \dot v_{i}
\end{equation}
for $\Theta_{i,j} \in 
(U^-_q
\mathfrak{sp}_{2\infty})_{-(\alpha_{|i|}+\alpha_{|i+1|}+\cdots+\alpha_{|j-1|})}$.
Each $\Theta_{i,j}$ lies in Lusztig's $\Z[q,q^{-1}]$-form
for $U_q^- \mathfrak{sp}_{2\infty}$
by the integrality of the quasi-$R$-matrix established in \cite[Corollary 24.1.6]{Lubook}.
Applying these remarks to (\ref{bardef}) and using induction, we deduce 
$\psi(\dot v_{\bb})$ equals $\dot v_{\bb}$
plus a $\Z[q,q^{-1}]$-linear combination of $\dot v_{\ba}$'s for
$\ba \succ \bb$, which is a well-defined element of
$\widehat{V}^{\otimes n}$.
The formula (\ref{bardef}) also makes sense for arbitrary sums 
$\sum_{\bb \in \B} p_\bb(q) \dot v_\bb$
due to the interval-finiteness of the inverse dominance order
on $P^n$.
Finally, to see that $\psi$ commutes with the
actions of all $\dot f_i$ and $\dot e_i$, and that it is an
involution, one argues as in \cite[$\S$27.3.1]{Lubook}.

As the following lemma shows, the various bar involutions we have defined are closely related.

\begin{lemma}\label{compatible}
The following diagrams commute:
$$
\begin{CD}
\widehat{V}^{\otimes \bsigma}_0 &@>\psi_0>>&\widehat{V}^{\otimes\bsigma}_0\\
@V\inc_0VV&&@AA\pr_0 A\\
\widehat{V}^{\otimes \bsigma} &@>>\psi>&\widehat{V}^{\otimes\bsigma}
\end{CD}\quad,
\qquad
\begin{CD}
\dot{V}^{\otimes n}_k &@>\psi_k>>&\dot{V}^{\otimes n}_k\\
@V\inc_kVV&&@AA\pr_k A\\
\widehat{V}^{\otimes n} &@>>\psi>&\widehat{V}^{\otimes n}
\end{CD}\quad,\qquad
\begin{CD}
\widehat{V}^{\otimes \bsigma}_0 &@>\psi_0>>&\widehat{V}^{\otimes\bsigma}_0\\
@V\inc_\bsigma VV&&@AA\pr_\bsigma A\\
\widehat{V}^{\otimes n} &@>>\psi>&\widehat{V}^{\otimes n}
\end{CD}\quad.
$$
\end{lemma}

\begin{proof}
In each case, this follows because the quasi-$R$-matrix $\Theta$
used to define the bottom map is a sum of the form
$\sum_{\beta} \Theta_\beta$ for $\beta$ in the positive root lattice
of $\mathfrak{sl}_\infty$ or $\mathfrak{sp}_{2\infty}$, while the
quasi-$R$-matrix used to define the top map
is a sum of the same $\Theta_\beta$'s for $\beta$ taken from the
positive root lattice of the subalgebra
$\mathfrak{sl}_{+\infty}$ or $\mathfrak{sp}_{2k}$.
\end{proof}

Now we can introduce the {\em canonical basis} for each of our completed
tensor spaces.
In each case, the bar involution maps the monomial basis vector
indexed by $\bb$ to itself plus a $\Z[q,q^{-1}]$-linear combination of
monomial basis vectors indexed by strictly larger 
$\ba$'s in the appropriate Bruhat order.
Then we apply
``Lusztig's Lemma''
as in the proof of \cite[Theorem 27.3.2]{Lubook}: the
canonical basis vector indexed by $\bb$ is the unique bar-invariant
vector that is equal to the monomial basis vector indexed by $\bb$
modulo a 
$q \Z[q]$-linear combination of other monomial basis vectors.
Our notation for the canonical basis in each case is explained in the next two paragraphs.

For $\mathfrak{sl}_\infty$, we denote the canonical basis for
$\widehat{V}^{\otimes \bsigma}$ as just defined by $\{\dot c_\bb^\bsigma\:|\:\bb \in
\B\}$. So, $\dot c_{\bb}^{\bsigma}$ is the unique vector fixed by
$\psi$ such that
\begin{equation}\label{pilly}
\dot c_{\bb}^{\bsigma} = \sum_{\ba \in \B}
d^\bsigma_{\ba,\bb}(q) \dot v_{\ba}^\bsigma
\end{equation}
for polynomials $d^\bsigma_{\ba,\bb}(q)$ with $d^\bsigma_{\bb,\bb}(q) = 1$,
$d^\bsigma_{\ba,\bb}(q) = 0$ unless $\ba \succeq \bb$, and $d^\bsigma_{\ba,\bb}(q) \in
q \Z[q]$ if $\ba \succ \bb$.
These polynomials have a natural representation theoretic
interpretation discussed in detail in \cite[$\S$5.9]{BLW}.
They are some finite type A parabolic Kazhdan-Lusztig polynomials (suitably
normalized),
hence, all of their coefficients are non-negative.
Moreover, each $\dot c_\bb^\bsigma$
is always a {\em finite} sum of $\dot v^\bsigma_\ba$'s,
i.e. $\dot c^\bsigma_\bb \in \dot V^{\otimes \bsigma}$
before completion.
We will not introduce any new notation for the canonical basis of
$\widehat{V}^{\otimes\bsigma}_0$ in the
$\mathfrak{sl}_{+\infty}$-case, because by the first diagram from Lemma~\ref{compatible} it is simply
the projection $\{\pr_0 \dot c_\bb^\bsigma\:|\:\bb \in \B_0\}$ of the
basis just defined.

Moving on to our notation for $\mathfrak{sp}_{2\infty}$, the
canonical basis for $\widehat{V}^{\otimes n}$ is
$\{\dot c_\bb\:|\:\bb \in \B\}$.
We have that 
\begin{equation}\label{polly}
\dot c_\bb = \sum_{\ba\in \B} d_{\ba,\bb}(q) 
\dot v_\ba
\end{equation}
for polynomials $d_{\ba,\bb}(q) \in \Z[q]$
with $d_{\bb,\bb}(q) = 1$, $d_{\ba,\bb}(q) = 0$ unless $\ba \succeq
\bb$,
and $d_{\ba,\bb}(q) \in q \Z[q]$ if $\ba \succ \bb$.
Unlike in the previous paragraph, the polynomials $d_{\ba,\bb}(q)$ may have
negative coefficients;
see Example~\ref{tsuchioka} below.
Consequently, it is conceivable that some $\dot c_\bb$'s
might fail to be {finite} sums of $\dot v_\ba$'s,
but this seems unlikely to us.
In view of the second diagram from Lemma~\ref{compatible},
the canonical basis for $\dot V^{\otimes n}_k$ is the projection
$\{\pr_k \dot c_\bb\:|\:\bb \in \B_k\}$.

The following lemma is an equivalent formulation of \cite[Proposition 4.1]{CKW}.

\begin{lemma}\label{ckwp}
For $\bb \in \B_\bsigma$, we have that
$\pr_\bsigma \dot c_{\bb} = \pr_0 \dot c_{\bb'}^\bsigma$.
Hence, $d_{\ba,\bb}(q) = d_{\ba',\bb'}^\bsigma(q)$
for all $\ba,\bb \in B_\bsigma$. 
\end{lemma}

\begin{proof}
As $\pr_\bsigma \dot v_{\bb} = \pr_0 \dot v_{\bb'}^{\bsigma}$,
 this follows using the third diagram from Lemma~\ref{compatible}.
\end{proof}

The vectors $c_\bb^\bsigma$ and
$c_\bb$ displayed in the table at the beginning of the subsection
refer to the
specializations of $\dot c_{\bb}^{\bsigma}$ and $\dot c_\bb$ at $q=1$.

\subsection{An algorithm}
In \cite[$\S$8]{BD}, we described an algorithm to compute the
canonical basis $\{\dot c_\bb^\bsigma\:|\:\bb \in \B\}$ for the $U_q \mathfrak{sl}_\infty$-module
$\widehat{V}^{\otimes\bsigma}$. To conclude the subsection, we work instead
with $U_q \mathfrak{sp}_{2\infty}$, and
describe
an analogous algorithm to compute the canonical basis 
$\{\dot c_\bb\:|\:\bb \in \B\}$
for $\widehat{V}^{\otimes
  n}$.
The algorithm goes by induction on $n$.
In case $n=1$, we have that $\dot c_\bb = \dot v_\bb$ always.
If $n > 1$, we begin by recursively 
computing $\dot c_{\bar\bb} \in \widehat{V}^{\otimes (n-1)}$, where
$\bar\bb$ denotes $(b_1,\dots,b_{n-1})$ as usual.
It is a linear combination of $\dot v_{\bar\ba}$'s for $\bar\ba \succeq \bar\bb$.
Then we define $j$ to be
 the greatest integer such that $j \leq b_n$, and $j \leq -|a_r|$
for all
$1 \leq r < n$ and all tuples
$\bar\ba = (a_1,\dots,a_{n-1})$ such that $\dot v_{\bar\ba}$ occurs
with non-zero coefficient in the expansion of $\dot c_{\bar\bb}$.

\begin{lemma}\label{lastofyear}
In the above notation, we have that
$\Theta\left(\dot c_{\bar \bb} \otimes \dot
v_j\right) = \dot c_{\bar\bb} \otimes \dot v_j$.
\end{lemma}

\begin{proof}
As in (\ref{rhs}), we have that
$\Theta\left(\dot c_{\bar \bb} \otimes \dot
v_j\right) = \dot c_{\bar\bb} \otimes \dot v_j
+ \sum_{i < j} \left(\Theta_{i,j} \dot c_{\bar\bb} \right) \otimes \dot v_i$,
where
$\Theta_{i,j}$ is a linear combination of non-trivial monomials
in $\dot f_{|j-1|}, \dot f_{|j-2|},\dots,\dot f_{|i|}$.
By the definition of $j$, all of these generators act as zero on $\dot c_{\bar\bb}$.
\end{proof}

Lemma~\ref{lastofyear} shows
that the vector 
$\dot c_{\bar \bb} \otimes \dot v_j\in
\widehat{V}^{\otimes n}$ is fixed by $\psi$.
Hence, so too is 
$\dot f_{|b_n-1|}\cdots \dot f_{|j+1|} \dot f_{|j|} 
\left(\dot c_{\bar \bb} \otimes \dot v_j\right)$.
By Lemma~\ref{not}, this new vector equals $\dot v_\bb$ plus a $\Z[q,q^{-1}]$-linear combination of $\dot
v_{\ba}$'s for $\ba \succ \bb$. If all but its leading coefficient lie in $q \Z[q]$, it is
already the desired vector $\dot c_{\bb}$. 
Otherwise, one picks $\ba
\succ \bb$ minimal so that the $\dot v_\ba$-coefficient is not in $q
\Z[q]$, then subtracts a bar-invariant multiple of the recursively
computed vector $\dot c_\ba$ to remedy this defficiency. Continuing in
this way, we finally obtain a bar-invariant vector with all of the required
properties to be
$\dot c_{\bb}$.

\begin{example}
The canonical basis of $V^{\otimes 2}$ consists of the following
vectors:
\begin{align*}
\dot v_i \otimes \dot v_j&\quad\text{for }i \geq j\text{ with }i+j \neq 1,\\
\dot v_i \otimes \dot v_j + q \dot v_j \otimes \dot v_i&\quad\text{for }i < j\text{ with
}i+j \neq 1,\\
\dot v_i \otimes \dot v_{1-i} + q \dot v_{1+i} \otimes \dot v_{-i}&\quad\text{for }i >
0,\\
\dot v_i \otimes \dot v_{1-i} 
+q \dot v_{i+1} \otimes \dot v_{-i}
+q \dot v_{-i} \otimes \dot v_{i+1}
+q^2
\dot v_{1-i} \otimes \dot v_i
&\quad\text{for }i < 0,\\
\dot v_0 \otimes \dot v_1 + q^2 \dot v_1 \otimes \dot v_0.&
\end{align*}
\end{example}

We refer the reader to {\tt http://pages.uoregon.edu/brundan/papers/C.gap}
for some
{\sc Gap} code implementing this algorithm.
Using it, we have independently verified the next examples, which were discovered originally by Tsuchioka:

\begin{example}\label{tsuchioka}
If $\ba = (1,1,0,1,0,0)$ and $\bb = (-1,2,-1,2,-1,2)$
then
$$
d_{\ba,\bb}(q) = q^7 + 4 q^5+ 3 q^3 - q.
$$
If $\ba = ( 1, -1, 2, -1, 2, 0 )$
and $\bb = (-1, -2, 3, -2, 3, 2
)$ then $$
d_{\ba,\bb}(q) = 8q^3-q.
$$ 
These examples demonstrate that positivity
fails in this situation.
\end{example}

\subsection{Crystals}\label{scr}
To conclude the section, we recall the explicit combinatorial description of
the crystal associated to the $\mathfrak{sp}_{2\infty}$-module
$V^{\otimes n}$. 
Later in the article, we will give a
representation-theoretic interpretation of this structure; see $\S$\ref{cryssec}.
The case of the $\mathfrak{sl}_\infty$-module $V^{\otimes\bsigma}$ can be treated
entirely similarly on replacing
$\iSIG(\bb)$ with $\iSIG^\bsigma(\bb)$; its representation-theoretic
significance is discussed e.g. in \cite[$\S$2.10]{BLW}.

The set underlying the crystal that we need is the set $\B$ that
parametrizes our various bases for $V^{\otimes n}$.
Its weight decomposition $\B = \bigsqcup_{\gamma \in P}
\B_\gamma$ is defined by setting
\begin{equation*}
\B_\gamma := \left\{ \bb \in \B  \:\big|\:\,|\!\WT(\bb)| =
  \gamma\right\}.
\end{equation*}
We need to introduce crystal operators
$$
\tilde f_i:\B_\gamma \rightarrow
\B_{\gamma-\alpha_i}\sqcup\{\varnothing\},
\qquad
\tilde e_i:\B_\gamma \rightarrow \B_{\gamma+\alpha_i} \sqcup
\{\varnothing\}
$$
for each $\gamma \in P$ and $i \in I$.
These arise naturally by 
iterating Kashiwara's tensor product rule, and may be computed
as follows.
Take $\bb \in \B_\gamma$.  Starting from 
the \textit{$i$-signature} $\iSIG(\bb)$ from (\ref{newisig}),
we define the \textit{reduced $i$-signature} by replacing pairs of entries
of the form  $\mathtt{e} \mathtt{f}$ (possibly separated by
$\bullet$'s) with $\bullet$'s, 
until all $\mathtt{e}$ entries appear to the right of
the entries $\mathtt{f}$.
Then define $\tilde f_i \bb$ to be 
$\bb + \bd_r$ if 
the rightmost $\mathtt{f}$
in the reduced $i$-signature appears in position $r$,
or
$\varnothing$ if there are no $\mathtt{f}$'s remaining in the reduced $i$-signature.  Similarly, define $\tilde e_i \bb$ to be 
$\bb - \bd_s$ if 
the leftmost $\mathtt{e}$
in the reduced $i$-signature appears in position $s$,
or
$\varnothing$ if there are no $\mathtt{e}$'s present. 

\begin{example} Take $\bb = (2,-1,-1,4,-2,-2,3,2,-2)$.  The $2$-signature of $\bb$ is the tuple
$(\mathtt{f, e, e, \bullet, f, f, e, f,f})$. The reduced
2-signature is $(\mathtt{f,\bullet, \bullet, \bullet, \bullet, \bullet, \bullet,
  \bullet, f})$.
Hence, $\tilde{f}_2\bb = \bb + \bd_9 = (2,-1,-1,4,-2,-2,3,2,-1)$
and
$\tilde{e}_2 \bb = \varnothing$.
\end{example}

Let $\B^\circ$ denote the 
set of all elements of $\B$ which can be obtained from $\bz = (0,\dots,0)$ by applying a sequence
of crystal operators.  In other words, $\B^\circ$ is the
\textit{connected component} of the crystal $\B$ containing $\bz$.  

\begin{lemma} \label{conncomp}
We have that $\bb \in \B^\circ$ if and only if $\bb$ is 
{\em antidominant}
in the sense that $b_1 \leq \cdots \leq b_n$. 
\end{lemma}

\begin{proof}
For the forward implication, we observe that whenever $\ba \in \B$ is antidominant, then so are
$\tilde{f}_i \ba$ and $\tilde{e}_i \ba$. 
For instance,
to check that the entries of $\tilde{f}_i \ba$ are weakly increasing,
we have that $\tilde{f}_i \ba = \ba + \bd_r$ where $r$ is the maximal
index for which the reduced $i$-signature of $\ba$ contains an 
$\mathtt{f}$.  
We need to see that $a_r < a_{r+1}$.
Well, otherwise, we would have that $a_r = a_{r+1}$, in which case
$\isig_r(\ba) = \isig _{r+1}(\ba) = \mathtt{f}$.  Because we cancel
$\mathtt{ef}$ pairs (and not $\mathtt{fe}$!) 
it would then follow that the reduced $i$-signature of $\ba$ contains
a $\mathtt{f}$ in its $(r+1)$th entry, which contradicts our assumption
about $r$.

Conversely, suppose that $b_1 \leq \cdots \leq b_n$.  For every index $r$, define a monomial
$$
\widetilde{x}_r := \left\{ \begin{array}{ll}  \tilde{f}_{b_r - 1} \cdots \tilde{f}_1\tilde{f}_0 & \text{ if $b_r \geq 0$} \\
							 \tilde{e}_{-b_r} \cdots \tilde{e}_2\tilde{e}_{1} & \text{ if $b_r < 0$.}
							 \end{array} \right.
$$
Letting $t$ denote the maximal index for which $b_t < 0$, taking $t :=
0$ in case $b_r \geq 0$ for all $r$, one then
checks that
$\widetilde{x}_t \cdots \widetilde{x}_2 \widetilde{x}_1
\widetilde{x}_{t+1}\widetilde{x}_{t+2} \cdots \widetilde{x}_n \bz = \bb.$
\end{proof}

Similarly, 
one can make the subset $\B_k \subset \B$ into an
$\mathfrak{sp}_{2k}$-crystal.  
The connected component of $\B_k$ containing $\bz$ is $\B_k^\circ :=\B_k
\cap \B^\circ$. It is also the connected component containing
\begin{equation}\label{zk}
\bz_k := (1-k,\dots,1-k) \in \B_k.
\end{equation}
This is significant because the vector $v_{\bz_k}$ is a highest weight
vector in $V_k^{\otimes n}$. 
Its weight $|\WT(\bz_k)|$ is $-n \eps_{k-1}$.

\section{Category $\O$}

Next, we introduce the supercategory $\sO$ of representations of the
Lie superalgebra $\mathfrak{q}_n(\K)$ that is the main object of
study of this article. Then, we prove our main categorification
theorem, which asserts that $\sO$ splits as $\O \oplus \Pi \O$ with
$\O$ being a tensor product categorification of the $\mathfrak{sp}_{2\infty}$-module
$V^{\otimes n}$. The proof of this theorem is similar to the proof of a
similar assertion for type A blocks from \cite{BD}.

\subsection{Superalgebra}
We will work from now on over the ground field $\K$.
A {\em vector superspace} is a $\Z/2$-graded vector space $V = V_\0
\oplus V_\1$.
We
denote the parity of a homogeneous vector $v \in V$
by $|v| \in \Z/2$.
Any $v \in V$ has a canonical
decomposition $v = v_\0 + v_\1$ with $|v_p|=p$.
Let $\underline{\SVec}$ be the category of vector superspaces and
parity-preserving linear maps. It is symmetric monoidal 
with braiding $u \otimes v \mapsto (-1)^{|u||v|} v \otimes u$.
Then, we make the following definitions following \cite{BE}:
\begin{itemize}
\item
A {\em supercategory} is a $\underline{\SVec}$-enriched category.
\item
 A {\em superfunctor} is a $\underline{\SVec}$-enriched functor.
\item
A \textit{supernatural transformation} $\eta : 
F \Rightarrow G$ 
between superfunctors
$F, G : \C \to \D$
is a family of morphisms 
$\eta_M = \eta_{M,\0} + \eta_{M,\1} : FM \to GM$ for each $M\in\ob \C$, such that
$\eta_{N,p}  \circ Ff = (-1)^{|f|p} Gf \circ \eta_{M,p}$
for every homogeneous morphism $f:M \rightarrow N$
in $\C$ and each $p \in \Z/2$.
\end{itemize}

For any supercategory $\C$, there is a supercategory
$\mathcal{E}\!nd(\C)$ consisting of superfunctors and supernatural
transformations.
It is a (strict) {\em monoidal supercategory} in the sense
of \cite[Definition 1.4]{BE}.
A {\em superequivalence} between supercategories $\C$
and $\D$ is a superfunctor $F:\C \to \D$ such that there exists
another superfunctor $G:\D \to \C$ with $GF:\C \to \C$ and $FG:\D \to \D$ being evenly
isomorphic to identity functors.

Given any $\K$-linear category $\C$, one can form the supercategory
$\C \oplus \Pi \C$ with objects being pairs $(V_1,V_2)$ of objects
from $\C$, and
morphisms 
$(V_1,V_2) \rightarrow (W_1,W_2)$ that are $2 \times
2$ matrices $f = 
\left( \begin{array}{cc} f_{11} & f_{12} \\ f_{21} & f_{22} \end{array} \right)$
of morphisms $f_{ij}\in \Hom_\C(W_j, V_i)$. The $\Z/2$-grading is
defined so $f_\0 = \left( \begin{array}{cc} f_{11} & 0 \\ 0 & f_{22} \end{array} \right)$
and $f_\1 = \left( \begin{array}{cc} 0 & f_{12} \\f_{21} & 0 \end{array}
\right)$.
We say that a supercategory {\em splits} if it is superequivalent to a
supercategory of this form.

Here is the basic example to keep in mind. 
Let $A = A_\0 \oplus A_\1$ be an associative superalgebra.
There is a supercategory $A\text{-}\SMod$
consisting of 
left $A$-supermodules.
Even morphisms in $A\text{-}\SMod$ are parity-preserving 
linear maps such that $f(av) = a f(v)$ for all $a \in A, v \in M$; odd morphisms are parity-reversing linear maps
such that $f(av) = (-1)^{|a|} a f(v)$ for homogeneous $a$.
If $A$ is purely even, i.e $A = A_\0$, then the category
$A\text{-}\SMod$ obviously splits as $A\text{-}\Mod\oplus \Pi(A\text{-}\Mod)$. In general,
$A\text{-}\SMod$ splits if and only if
$A$ is Morita superequivalent to
a purely even superalgebra.

\subsection{\boldmath Supercategory $\sO$}
We assume henceforth that we have fixed $n\geq 1$, and set $m :=
\lceil n/2\rceil$.
We are interested in a certain supercategory of representations of the Lie
superalgebra $\mathfrak{q}_n(\K)$, that is, the
subalgebra of the general linear Lie superalgebra
$\mathfrak{gl}_{n|n}(\K)$ consisting of matrices of the form
$\left( \begin{array}{c|c} A & B \\\hline B & A \end{array} \right)$.
In order to unify our treatment of odd versus even $n$ as much as
possible, we will adopt the same trick as used in \cite{BD}, setting
$$
\g = \g_\0 \oplus \g_\1:= \left\{
\begin{array}{ll}
\mathfrak{q}_n(\K)&\text{if $n$ is even},\\
\mathfrak{q}_n(\K)\oplus \mathfrak{q}_1(\K)&\text{if $n$ is odd.}
\end{array}
\right.
$$
The point of the additional $\mathfrak{q}_1(\K)$ in case $n$ is odd
is that it adjoins an extra odd involution to the supercategory $\sO$
to be defined shortly. In language from the introduction of
\cite{BD}, this amounts to working with the {\em Clifford twist} of
the supercategory that one would naturally define without this extra factor.

It will sometimes be helpful to identify $\g$ with a subalgebra of $\widehat\g := \mathfrak{gl}_{2m|2m}(\K)$.
Let
$x_{r,s}$ be the usual $rs$-matrix unit in $\widehat\g$,
which is even if $1 \leq r,s \leq 2m$ or $2m+1 \leq r,s \leq 4m$, and
odd otherwise.
Introduce the matrices
\begin{align}\label{a1}
e_{r,s} &:= x_{r,s} + x_{2m+r,2m+s},
&
e'_{r,s} &:= x_{r,2m+s}+x_{2m+r,s},\\
f_{r,s} &:= x_{r,s} - x_{2m+r,2m+s},
&
f'_{r,s} &:=  x_{r,2m+s} - x_{2m+r,s},\label{a2}\\
d_r&:= e_{r,r},
&d_r'&:= e_{r,r}'.
\end{align}
Then $\g$ is the subalgebra of $\widehat\g$ with basis 
$\{e_{r,s}, e_{r,s}'\:|\:1 \leq r,s \leq n\}$ together with
$\{d_{2m}, d_{2m}'\}$ if $n$ is odd.
The matrices $f_{r,s}$, $f_{r,s}'$ are elements of $\widehat\g$ but
not $\g$.
Let $\h = \h_\0 \oplus \h_\1$ be the Cartan subalgebra of $\g = \g_0\oplus\g_\1$ with basis
$\{d_r, d_r'\:|\:1 \leq r \leq 2m\}$.
Also let 
$\delta_1,\dots,\delta_{2m}$ be the basis for $\h_\0^*$ that is dual to
the basis $d_1,\dots,d_{2m}$ for $\h_0$.  
Finally, let $\b$ be the Borel subalgebra of $\g$
generated by $\h$ and the matrices
$\{e_{r,s}, e_{r,s}'\:|\:1 \leq r < s \leq n\}$.

As in the previous section, $\B$ will denote the set $\Z^n$ of
$n$-tuples $\bb = (b_1,\dots,b_n)$ of integers.
For $\bb \in \B$, 
let
$\lambda_\bb \in \h_\0^*$
be the weight defined from
\begin{equation}\label{wtdict2}
\lambda_\bb := 
\left\{
\begin{array}{ll}
\displaystyle\sum_{r=1}^n (b_r-\half)\delta_r&\text{if $n$ is
  even,}\\
\displaystyle\sum_{r=1}^n (b_r-\half)\delta_r  + \delta_{2m}&\text{if $n$ is
  odd.}
\end{array}
\right.
\end{equation}
Then we define $\sO$ to be the supercategory consisting of all 
$\g$-supermodules $M$ such that
\begin{itemize}
\item $M$ is finitely generated over $\g$;
\item $M$ is locally finite-dimensional over $\b$;
\item $M$ is semisimple over $\h_\0$ with all weights of the form
$\lambda_\bb$ for $\bb \in \B$.
\end{itemize}
We denote the usual {\em parity switching functor} by
$\Pi:\sO \rightarrow \sO$. This sends a supermodule $M$ to the
same vector space viewed as a superspace with $(\Pi M)_\0 := M_\1$ and
$(\Pi M)_\1 := M_\0$, and new action defined from $x \cdot v := (-1)^{|x|} xv$.

Let $\underline{\sO}$ be the underlying $\K$-linear category consisting of all of the
same objects as $\sO$, but only the even morphisms. 
The category $\underline{\sO}$ is obviously Abelian.
In fact, it is {\em Schurian} in following sense; 
this follows as in \cite[Lemma 2.3]{Btilt}.

\begin{definition}\label{schurian}
A $\CC$-linear category is {\em Schurian} if it is Abelian,
all of its objects are of finite
length, the endomorphism algebras of the irreducible objects are
one-dimensional,
and there are enough projectives and injectives.
\end{definition}

We proceed to introduce the Verma supermodules in $\sO$.
We need to do this rather carefully in order to be able
to distinguish a Verma supermodule from its parity flip.
Since we reserve the letter $i$ for elements of the set $I$ as in the 
previous section, we'll denote the canonical element of $\K$ by
$\sqrt{-1}$.
We also need to pick some distinguished square roots for each element of 
the subset $\Z+\half$ of $\K$
such that
\begin{equation}\label{needy}
\sqrt{i+\half}\sqrt{i-\half} = \sqrt{-i+\half}\sqrt{-i-\half}
\end{equation}
for each $i \in \N$.
For example, this can be done by 
letting $\sqrt{i+\half}$ denote the usual positive square root when $i
\geq 0$, then setting
$\sqrt{i+\half} := (-1)^{i+1} \sqrt{-1} \sqrt{-i-\half}$
if $i < 0$.

\begin{lemma}
For each $\bb \in \B$, there is a unique (up to even isomorphism)
irreducible $\h$-supermodule $V(\bb)$ 
of weight $\lambda_\bb$
such that the element
$d_1'\cdots d_{2m}' \in U(\g)$ acts
on all even (resp. odd) vectors in $V(\bb)$ by multiplication by the
scalar
$c_\bb$ (resp. $-c_\bb$), where
\begin{equation}\label{y1}
c_\bb
:=
(\sqrt{-1})^m \sqrt{b_1-\half}\cdots \sqrt{b_n-\half}.
\end{equation}
Moreover, any $\h$-supermodule of weight $\lambda_\bb$
splits as a direct sum of copies of $V(\bb)$ and its parity flip $\Pi V(\bb)$.
\end{lemma}

\begin{proof}
This is similar to \cite[Lemma 2.1]{BD}.
The supermodule $V(\bb)$ may be constructed explicitly as there as an
irreducible supermodule over a Clifford superalgebra of  rank $2m$;
in particular, $\dim V(\bb) = 2^m$.
\end{proof}

For each $\bb \in \B$, we define the {\em Verma supermodule}
\begin{equation}
M(\bb) := U(\g) \otimes_{U(\b)} V(\bb),
\end{equation}
viewing $V(\bb)$ as a $\b$-supermodule via the natural surjection $\b
\twoheadrightarrow \h$.
It is obvious that this belongs to $\sO$.
Here we list some more basic facts.
\begin{itemize}
\item
The Verma supermodule $M(\bb)$ has a unique irreducible quotient
$L(\bb)$. The supermodules
$\{L(\bb)\:|\:\bb \in \B\}$ give
a complete set of representatives for the isomorphism classes of irreducible objects in
$\sO$. Moreover, $L(\bb)$ is not
evenly isomorphic to its parity flip.
\item
There is a duality $\star$ on $\sO$ 
such that $L(\bb)$ and
$L(\bb)^\star$ are evenly isomorphic for each $\bb \in \B$; 
cf.
\cite[Lemma 2.3]{BD}.
\item
If $\bb$ is both {\em dominant} in the sense that $b_1 \geq \cdots
\geq b_n$, and {\em typical}, meaning that $b_r+b_s \neq 1$ for
all $1 \leq r < s \leq n$, then $M(\bb)$ is projective;
cf. \cite[Lemma 2.4]{BD}.
\end{itemize}

Let $\sO^\Delta$ be the full subcategory of $\sO$ consisting of all
supermodules possessing a Verma flag,
i.e. for which there is a filtration $0 = M_0 \subset \cdots \subset M_l = M$ 
with sections $M_k / M_{k-1}$
that are isomorphic  to Verma supermodules.
As in \cite[Lemma 2.5]{BD}, the multiplicities $(M:M(\bb))$ and
$(M:\Pi M(\bb))$
of $M(\bb)$ and $\Pi M(\bb)$
in any Verma flag of $M\in \ob\sO^\Delta$ satisfy
\begin{align*}
(M:M(\bb)) &= \dim \Hom_{\sO}(M, M(\bb)^\star)_\0,\\
(M:\Pi M(\bb)) &= \dim \Hom_{\sO}(M, M(\bb)^\star)_\1.
\end{align*}
Moreover, if $M$ possesses a Verma flag, then so does any direct
summand of $M$.

\subsection{Special projective superfunctors}
Let $\widehat{U}$ be the natural $\widehat{\mathfrak{g}}$-supermodule
of column vectors with standard basis $u_1,\dots,u_{2m},
u_1',\dots,u_{2m}'$, so the unprimed vectors are even, the primed ones
are odd.
Let $\widehat{U}^*$ be its dual, with 
basis $\phi_1,\dots,\phi_{2m}, \phi_1',\dots,\phi_{2m}'$
that is dual to the basis $u_1,\dots,u_{2m}, u_1',\dots,u_{2m}'$. 
Then, let $U \subseteq \widehat{U}$ and $U^* \subseteq \widehat{U}^*$
be the $\g$-supermodules
with 
bases $u_1,\dots,u_n,u_1',\dots,u_n'$ and
$\phi_1,\dots,\phi_n,\phi_1',\dots,\phi_n'$, respectively.

It is easy to see that tensoring either with $U$ or with $U^*$ sends supermodules
in $\sO$ to supermodules in $\sO$. Hence, we have endofunctors
\begin{equation}\label{spc}
\sF := U \otimes -:\sO \rightarrow \sO,
\qquad
\sE := U^* \otimes -:\sO \rightarrow \sO.
\end{equation}
The superfunctors $\sF$ and $\sE$ are both left and right adjoint to
each other.
The canonical adjunction making $(\sE, \sF)$ into an adjoint pair 
is induced by the linear maps
$$
U^* \otimes U \rightarrow \K,\:
\phi \otimes u \mapsto \phi(u),
\qquad
\K \rightarrow U \otimes U^*,\:
1 \mapsto \sum_{r=1}^n (u_r \otimes \phi_r + u_r' \otimes \phi_r'),
$$
while the adjunction $(\sF, \sE)$
is induced by 
$$
U \otimes U^* \rightarrow \K,\,
u \otimes \phi \mapsto (-1)^{|\phi||u|}\phi(u),
\quad
\K \rightarrow U^* \otimes U,\,
1 \mapsto \sum_{r=1}^n (\phi_r \otimes u_r - \phi_r' \otimes u_r').
$$
As well as these adjunctions, 
there are even
supernatural transformations
$x:\sF \Rightarrow \sF$ and
$t:\sF^2 \Rightarrow \sF^2$, and an odd supernatural transformation
$c:\sF \Rightarrow \sF$, which are
defined on $M \in \ob \sO$ as follows:
\begin{itemize}
\item 
$x_M:U \otimes M \rightarrow U \otimes M$ 
is left multiplication by the tensor 
\begin{equation*}
\omega := 
\sum_{r,s=1}^n \left(f_{r,s}\otimes e_{s,r}  -
  f_{r,s}'\otimes e_{s,r}'\right)
\in \widehat{\g} \otimes \g,
\end{equation*}
which defines a $\g$-supermodule homomorphism by the proof of
\cite[Lemma 3.1]{BD};
\item
$t_M:U \otimes U \otimes M \rightarrow U \otimes U \otimes M$ 
sends
$u \otimes v \otimes m \mapsto (-1)^{|u||v|} v \otimes u
\otimes m$;
\item
$c_M: U \otimes M \rightarrow U \otimes M$
is left  multiplication by $\sqrt{-1} \,z' \otimes 1$
where
\begin{equation*}
z' := \sum_{t=1}^n f_{t,t}'
\in \widehat\g.
\end{equation*}
\end{itemize}
Similarly, there are supernatural transformations
$x^*:\sE \Rightarrow \sE, t^*:\sE^2 \Rightarrow \sE^2$ and
$c^*:\sE \Rightarrow \sE$: $x^*$ and $c^*$ are defined
similarly to $x$ and $c$ but with an additional sign, so they are given 
by left multiplication by $- \omega$ and by $-\sqrt{-1}\,z'
\otimes 1$, respectively; $t^*$ is defined using the braiding on
$\SVec$ in exactly the same
way as $t$.
One can check that $x^*, t^*$ and $c^*$ are both the left and right mates
of $x, t$ and $c$, respectively, with respect to the adjunctions fixed in the previous
paragraph; cf. \cite[Lemma 3.6]{BD}.

\begin{definition}\label{AHC}
The (degenerate) {\em affine Hecke-Clifford supercategory} $\AHC$ is the 
strict monoidal supercategory
with a single generating object $1$, even generating morphisms
$\mathord{
\begin{tikzpicture}[baseline = -1]
	\draw[-,thick,darkblue] (0.08,-.15) to (0.08,.3);
      \node at (0.08,0.08) {$\color{darkblue}\bullet$};
\end{tikzpicture}
}:1 \rightarrow 1$ and
$\mathord{
\begin{tikzpicture}[baseline = -1]
	\draw[-,thick,darkblue] (0.18,-.15) to (-0.18,.3);
	\draw[-,thick,darkblue] (-0.18,-.15) to (0.18,.3);
\end{tikzpicture}
}:1 \otimes 1 \rightarrow 1 \otimes 1$,
and an odd generating morphism
$\mathord{
\begin{tikzpicture}[baseline = -1]
	\draw[-,thick,darkblue] (0.08,-.15) to (0.08,.3);
      \node at (0.08,0.08) {$\color{darkblue}\circ$};
\end{tikzpicture}
}:1 \rightarrow 1$,
subject to the following relations:
\begin{align*}
%xc = - c x
\mathord{
\begin{tikzpicture}[baseline = 2]
	\draw[-, thick,darkblue] (0, 0.6) to (0, -0.3);
	\node at (0, -0.00) {$\color{darkblue} \circ$};
	\node at (0, 0.3) {$\color{darkblue} \bullet$};
\end{tikzpicture}
} &= 
- 
\mathord{
\begin{tikzpicture}[baseline = 2]
	\draw[-, thick,darkblue] (0, 0.6) to (0, -0.3);
	\node at (0, -0.0) {$\color{darkblue} \bullet$};
	\node at (0, 0.3) {$\color{darkblue} \circ$};
\end{tikzpicture}
}\:,
% c^2 = 1
&\mathord{
\begin{tikzpicture}[baseline = 2]
	\draw[-, thick,darkblue] (0, 0.6) to (0, -0.3);
	\node at (0, -0.0) {$\color{darkblue} \circ$};
	\node at (0, 0.3) {$\color{darkblue} \circ$};
\end{tikzpicture}
}
&= 
\mathord{
\begin{tikzpicture}[baseline = 2]
	\draw[-, thick,darkblue] (0, 0.6) to (0, -0.3);
\end{tikzpicture}
}\:,
&
%t^2 = 1
\mathord{
\begin{tikzpicture}[baseline = 8.5]
	\draw[-,thick,darkblue] (0.2,.4) to[out=90,in=-90] (-0.2,.9);
	\draw[-,thick,darkblue] (-0.2,.4) to[out=90,in=-90] (0.2,.9);
	\draw[-,thick,darkblue] (0.2,-.1) to[out=90,in=-90] (-0.2,.4);
	\draw[-,thick,darkblue] (-0.2,-.1) to[out=90,in=-90] (0.2,.4);
\end{tikzpicture}}
 &= 
\mathord{
\begin{tikzpicture}[baseline = 8.5]
	\draw[-,thick,darkblue] (0.1,-.1) to (0.1,.9);
	\draw[-,thick,darkblue] (-0.3,-.1) to (-0.3,.9);
\end{tikzpicture}
}\:,
\end{align*}\begin{align*}
%c and t (Only need one because t ^1 = )
\mathord{
\begin{tikzpicture}[baseline = 4]
	\draw[-,thick,darkblue] (0.25,.6) to (-0.25,-.2);
	\draw[-,thick,darkblue] (0.25,-.2) to (-0.25,.6);
      \node at (-0.14,0.42) {$\color{darkblue}\circ$};
\end{tikzpicture}
}
 &= 
\mathord{
\begin{tikzpicture}[baseline = 4]
	\draw[-,thick,darkblue] (0.25,.6) to (-0.25,-.2);
	\draw[-,thick,darkblue] (0.25,-.2) to (-0.25,.6);
      \node at (0.14,-0.02) {$\color{darkblue}\circ$};
\end{tikzpicture}
}
&%x and t relation
\mathord{
\begin{tikzpicture}[baseline = 4]
	\draw[-,thick,darkblue] (0.25,.6) to (-0.25,-.2);
	\draw[-,thick,darkblue] (0.25,-.2) to (-0.25,.6);
      \node at (-0.14,0.42) {$\color{darkblue}\bullet$};
\end{tikzpicture}
}
-
\mathord{
\begin{tikzpicture}[baseline = 4]
	\draw[-,thick,darkblue] (0.25,.6) to (-0.25,-.2);
	\draw[-,thick,darkblue] (0.25,-.2) to (-0.25,.6);
      \node at (0.14,-0.02) {$\color{darkblue}\bullet$};
\end{tikzpicture}
}
&= 
\mathord{
\begin{tikzpicture}[baseline = 4]
 	\draw[-,thick,darkblue] (0.08,-.2) to (0.08,.6);
	\draw[-,thick,darkblue] (-0.28,-.2) to (-0.28,.6);
\end{tikzpicture}
} 
- \!
\mathord{
\begin{tikzpicture}[baseline = 4]
 	\draw[-,thick,darkblue] (0.08,-.2) to (0.08,.6);
	\draw[-,thick,darkblue] (-0.28,-.2) to (-0.28,.6);
	\node at (0.08, 0.2) {$\color{darkblue}\circ$};
	\node at (-0.28, 0.2){$\color{darkblue}\circ$};
\end{tikzpicture}
}\:,
&%braid relation
\mathord{
\begin{tikzpicture}[baseline = 3]
	\draw[-,thick,darkblue] (0.45,.8) to (-0.45,-.4);
	\draw[-,thick,darkblue] (0.45,-.4) to (-0.45,.8);
        \draw[-,thick,darkblue] (0,-.4) to[out=90,in=-90] (-.45,0.2);
        \draw[-,thick,darkblue] (-0.45,0.2) to[out=90,in=-90] (0,0.8);
\end{tikzpicture}
}
&=
\mathord{
\begin{tikzpicture}[baseline = 3]
	\draw[-,thick,darkblue] (0.45,.8) to (-0.45,-.4);
	\draw[-,thick,darkblue] (0.45,-.4) to (-0.45,.8);
        \draw[-,thick,darkblue] (0,-.4) to[out=90,in=-90] (.45,0.2);
        \draw[-,thick,darkblue] (0.45,0.2) to[out=90,in=-90] (0,0.8);
\end{tikzpicture}
}\,.
\end{align*}
(Here, we are using the string calculus for strict monoidal
supercategories as in \cite{BE}.)
\end{definition}

The following theorem is essentially 
\cite[Theorem 7.4.1]{HKS}; cf. \cite[Theorem 6.2]{BD}.
It is proved by explicitly checking the relations.

\begin{theorem}\label{hkst}
There is a strict monoidal superfunctor
$\Psi:\AHC \rightarrow \mathcal{E}nd(\sO)$
sending the generating object
$1$ to the endofunctor $\sF$, and the generating morphisms 
$\mathord{
\begin{tikzpicture}[baseline = -1]
	\draw[-,thick,darkblue] (0.08,-.15) to (0.08,.3);
      \node at (0.08,0.08) {$\color{darkblue}\bullet$};
\end{tikzpicture}
},
$
$\mathord{
\begin{tikzpicture}[baseline = -1]
	\draw[-,thick,darkblue] (0.18,-.15) to (-0.18,.3);
	\draw[-,thick,darkblue] (-0.18,-.15) to (0.18,.3);
\end{tikzpicture}
}$
and
$\mathord{
\begin{tikzpicture}[baseline = -1]
	\draw[-,thick,darkblue] (0.08,-.15) to (0.08,.3);
      \node at (0.08,0.08) {$\color{darkblue}\circ$};
\end{tikzpicture}
}$
to the
supernatural transformations
$x, t$ and $c$, respectively.
\end{theorem}

The superfunctor $\Psi$ from Theorem~\ref{hkst} induces superalgebra
homomorphisms
\begin{equation}\label{tea2}
\Psi_d:AHC_d \rightarrow \End(\sF^d)
\end{equation}
for each $d \geq 0$, where
$AHC_d$ denotes the
(degenerate) {\em affine Hecke-Clifford superalgebra}
\begin{equation}\label{ahca}
AHC_d := \End_{\AHC}(1^{\otimes d}).
\end{equation}
These superalgebras were introduced originally in
\cite{N}, and can be understood
algebraically as follows.
Numbering the strings of a $d$-stringed diagram by $1,\dots,d$ from right to left,
let $x_r$ (resp. $c_r$) denote the element of
$AHC_d$ 
defined by a closed dot (resp. an open dot) on the $r$th
string. Let $t_r$ denote the crossing of the $r$th and $(r+1)$th
string.
The even elements $x_1,\dots, x_d$ commute, the odd elements
$c_1,\dots,c_d$ satisy the relations $c_r^2=1$ and $c_r c_s = -c_s c_r
\:(r \neq s)$ of the rank $d$ Clifford superalgebra $C_d$,
and $t_1,\dots,t_{d-1}$ satisfy the same relations as the
basic transpositions in the symmetric group $S_d$.
In fact, by the basis theorem for $AHC_d$ from \cite[$\S$2-k]{BK}, 
$x_1,\dots,x_d$ generate a copy of the polynomial algebra
$A_d := \K[x_1,\dots,x_d]$
inside $AHC_d$, while $c_1,\dots,c_d, t_1,\dots,t_{d-1}$ generate a
copy of the {\em Sergeev superalgebra}
$HC_d := S_d \ltimes C_d.$
Moreover, the natural multiplication map
$HC_d \otimes A_d \rightarrow AHC_d$ is an isomorphism of vector
superspaces.
We note also that
the
multiplication in $AHC_d$ satisfies the following:
\begin{align}\label{mirror1}
  f\, c_r &= c_r \,c_r(f),\\
 f \, t_r &= t_r\, t_r(f) + \partial_r(f)
+c_r c_{r+1} \tilde\partial_r(f),\label{mirror2}
\end{align}
for each $f \in A_d$. Here, the operators $c_r, t_r, \partial_r,\tilde\partial_r:A_d \rightarrow
A_d$ are defined as follows:
\begin{itemize}
\item $t_r$ is the automorphism that interchanges $x_r$ and
  $x_{r+1}$ and
fixes all other generators;
\item
$c_r$ is the automorphism that sends $x_r \mapsto -x_r$ and
fixes all other generators;
\item
$\partial_r$ is the Demazure operator $\partial_r(f) :=
\frac{t_r(f)-f}{x_r - x_{r+1}};
$
\item
$\tilde{\partial}_r$ is the {twisted Demazure operator} $c_{r+1}
\circ \partial_r \circ c_{r}$.
\end{itemize}

Next, we are going to decompose
$\sF$ and $\sE$ into generalized eigenspaces with respect to the
endomorphisms $x$ and $x^*$.
The key ingredient needed to understand this is the following, whose proof is
identical to that of \cite[Lemma 3.2]{BD}.

\begin{lemma} \label{eigenvalues} Suppose that $\bb \in \B$
and let $M := M(\bb)$.
\begin{enumerate}
\item 
There is 
a filtration $$
0 = M_0 \subset M_1 \subset \cdots \subset M_n = 
U \otimes M
$$ 
with
$M_{t} / M_{t-1} \cong M(\bb + \bd_t) \oplus \Pi M(\bb + \bd_t)$
for each $t=1,\dots,n$.  The endomorphism
$x_{M}$ preserves this filtration,
and the induced endomorphism
of the subquotient $M_t / M_{t-1}$ is diagonalizable
 with exactly two 
eigenvalues $\pm \sqrt{b_t+\half}\sqrt{b_t-\half}$.
Its $\sqrt{b_t+\half}\sqrt{b_t-\half}$-eigenspace is
 evenly isomorphic to $M(\bb + \bd_t)$,
while the other eigenspace is evenly isomorphic to 
$\Pi M(\bb + \bd_t)$.
\item 
There is 
a filtration 
$$
0 = M^{n} \subset \cdots \subset M^1 \subset M^0 = 
U^*
\otimes  M
$$ 
with
$M^{t-1} / M^{t} \cong M(\bb - \bd_t) \oplus \Pi M(\bb - \bd_t)$
for each $t=1,\dots,n$.  The endomorphism
$x^*_{M}$ preserves this filtration,
and the induced endomorphism
of the subquotient $M^{t-1} / M^{t}$ is diagonalizable
 with exactly two 
eigenvalues $\pm \sqrt{b_t-\half}\sqrt{b_t-\threehalves}$.
Its $\sqrt{b_t-\half}\sqrt{b_t-\threehalves}$-eigenspace is
 evenly isomorphic to $M(\bb - \bd_t)$,
while the other eigenspace is evenly isomorphic to 
$\Pi M(\bb - \bd_t)$.
 \end{enumerate}
\end{lemma}

{\em For the remainder of the section}, we let $I$ denote the set $\N$. In the notation from the
previous section, this is the index set for
the simple roots of the Kac-Moody algebra $\mathfrak{s} =
\mathfrak{sp}_{2\infty}$.
Let
\begin{equation}
J :=
\left\{\pm \sqrt{i+\half}\sqrt{i-\half}\:\bigg|\:i \in I\right\}.
\end{equation}
This set is relevant due to the following lemma.

\begin{lemma}\label{minpoly}
For any $M \in \ob\sO$, all roots of the minimal polynomials of
$x_M$ and $x_M^*$
(computed
in the finite dimensional superalgebras
$\End_{\sO}(\sF\,M)$ and $\End_{\sO}(\sE\,M)$)
belong to the set $J$.
\end{lemma}

\begin{proof}
This reduces to the case that $M$
is a Verma supermodule, when it follows from
Lemma~\ref{eigenvalues} and (\ref{needy}).
\end{proof}

For $j \in J$, let $\sF_j$ (resp.\ $\sE_j$) be the subfunctor of $\sF$
(resp. $\sE$) defined
by letting $\sF_j\, M$
(resp. $\sE_j\, M$) be the generalized $j$-eigenspace of
$x_M$ (resp. $x_M^*$)
for each $M \in \ob\sO$.
Lemma~\ref{minpoly} implies that
\begin{equation}\label{cc}
\sF = \bigoplus_{j \in J} \sF_j,
\qquad
\sE = \bigoplus_{j \in J} \sE_j.
\end{equation}
The adjunctions
$(\sE, \sF)$ and $(\sF, \sE)$ fixed earlier restrict to adjunctions
$(\sE_j, \sF_j)$ and $(\sF_j, \sE_j)$
for each $j \in J$; this follows because
$x^*$ is both the left and right mate of $x$.
Also, by Theorem~\ref{hkst}, 
$c$ restricts to an odd isomorphism 
$\sF_j \stackrel{\sim}{\Rightarrow} \sF_{-j}$ for each $j \in J$;
similarly, $\sE_j \cong \sE_{-j}$.

Recalling (\ref{cg}),
the following theorem reveals the first significant connection between
combinatorics in $\sO$ and the $\mathfrak{sp}_{2\infty}$-module
$V^{\otimes n}$.

\begin{theorem}\label{maintfthm2}
Given $\bb \in \B$ and $i \in I$,
let $j := 
\sqrt{i+\half}\sqrt{i-\half}$.
Then:
\begin{itemize}
\item[(1)]
$\sF_j \,M(\bb)$ (resp. $\sF_{-j}\,M(\bb)$) has a multiplicity-free 
filtration with sections that are 
evenly (resp. oddly) isomorphic to the Verma supermodules
$$
\{M(\bb+\bd_t)\:|\:\text{for }1 \leq t \leq n\text{ such that }\isig_t(\bb) =
\mathtt{f}\},
$$
appearing from bottom to top in order of increasing $t$.
\item[(2)]
$\sE_j \,M(\bb)$ (resp. $\sE_{-j}\,M(\bb)$) has a multiplicity-free 
filtration with sections that are 
evenly (resp. oddly) isomorphic to the Verma supermodules
$$
\{M(\bb-\bd_t)\:|\:\text{for }1 \leq t \leq n\text{ such that
}\isig_t(\bb) = \mathtt{e}\},
$$
appearing from top to bottom in order of increasing $t$.
\end{itemize}
\end{theorem}

\begin{proof}
(1) We just need to check the statement for $\sF_j\, M(\bb)$; the one
about
$\sF_{-j}\,M(\bb)$ then follows because it is isomorphic to
$\sF_j\,M(\bb)$ via an odd isomorphism.
Applying Lemma~\ref{eigenvalues}, we see that
$\sF_j \,M(\bb)$ has a multiplicity-free filtration with sections that
are evenly isomorphic to the supermodules
$M(\bb+\bd_t)$
for $t=1,\dots,n$ such that
$\sqrt{b_t+\half}\sqrt{b_t-\half} = j = \sqrt{i+\half}\sqrt{i-\half}$.
Squaring both sides, we deduce that
$b_t^2 = i^2$, hence,
$b_t = \pm i$. Both cases do indeed give
solutions thanks to (\ref{needy}). It remains to compare what we have
proved with the definition
of $i$-signature from (\ref{newisig}).

(2) Similar.
\end{proof}

Finally in this subsection, we introduce a completion
$\widehat{AHC}_d$ of the affine
Hecke-Clifford superalgebra $AHC_d$ from (\ref{ahca}), following
\cite[Definition 5.3]{KKT}.
As a vector superspace, we have that
\begin{equation}\label{completed}
\widehat{AHC}_d := HC_d \otimes \widehat{A}_d\qquad
\text{where}\qquad
\widehat{A}_d := \bigoplus_{\bj \in J^d} 
\K[[x_1-j_1,\dots,x_d-j_d]]1_\bj,
\end{equation}
and $J^d$ denotes the set of $d$-tuples $\bj = j_d\cdots j_1$ of
elements of $J$.
For $h \in HC_d$ and $f \in \K[[x_1-j_1,\dots,x_d-j_d]]$, we write simply $h f
1_\bj$ in
place of $h \otimes f 1_\bj$.
The multiplication in $\widehat{AHC}_d$
is defined so that $\widehat{A}_d$ is a subalgebra,
the maps $HC_d \hookrightarrow \widehat{AHC}_d,
h \mapsto h 1_\bj$ are algebra homomorphisms,
and,
extending (\ref{mirror1})--(\ref{mirror2}),
we have that:
\begin{align}\label{opaque1}
  (f 1_{\bj}) \, (c_r 1_{\bj'})
&= c_r\, c_r(f)  1_{c_r(\bj)} 1_{\bj'},\\\notag
(f 1_{\bj}) \, (t_r 1_{\bj'}) &=  t_r \, t_r(f) 1_{t_r(\bj)}1_\bj' 
+ \frac{t_r(f) 1_{t_r(\bj)} 
- f 1_\bj}{x_r-x_{r+1}}1_{\bj'}\\
&\qquad\qquad\qquad+c_r c_{r+1} \frac{
t_r(f)  1_{t_r(\bj)}- c_{r+1}(c_{r}(f))  1_{c_{r+1}(c_r(\bj))}}{x_r + x_{r+1}} 1_{\bj'}.\label{opaque2}
\end{align}
Let $\End(\sF^d)$ be the superalgebra of all supernatural
transformations $\sF^d \Rightarrow \sF^d$.
Since $(x-j)$ acts {\em locally nilpotently} on $\sF_j$,
i.e. it induces a nilpotent endomorphism of $\sF_j \,M$ for each $M \in
\ob \sO$, we can extend the homomorphism
$\Psi_d$ from (\ref{tea2}) uniquely to a homomorphism 
\begin{equation}\label{tea}
\widehat{\Psi}_d:\widehat{AHC}_d \rightarrow \End(\sF^d)
\end{equation}
such that $\widehat{\Psi}_d(1_\bj)$ is the projection of $\sF^d$ onto
its summand $\sF_{j_d}\,\cdots\,\sF_{j_1}$,
and $\widehat{\Psi}_d(a 1_\bj) = \Psi_d(a) \circ \widehat{\Psi}_d(1_\bj)$
for each $a \in AHC_d$.

\subsection{Indecomposable projectives}
In this subsection, we relate the $\mathfrak{sp}_{2\infty}$-Bruhat order $\preceq$ on $\B$
from $\S$\ref{bos} to the structure of the Verma supermodules in $\sO$.
Actually, it is better to work in terms of projectives, so
let $P(\bb)$ be a projective cover of $L(\bb)$ in $\underline{\sO}$.

\begin{theorem}\label{mainsplitthm2} 
The indecomposable projective supermodule $P(\bb)$ 
has a Verma flag with top section evenly isomorphic to $M(\bb)$
and other
sections 
evenly isomorphic to $M(\bc)$'s for $\bc \in \B$ with $\bc \succ \bb$.
\end{theorem}

\begin{proof}
By Lemma~\ref{construction2},
there exists a dominant, typical $\ba \in \B$ and a monomial $X$ in the
Chevalley generators
$\{f_i\:|\:i \in I\}$ of $\mathfrak{sp}_{2\infty}$
such that
$X v_\ba = v_\bb + 
(\text{a sum of $v_\bc$'s
for $\bc \succ \bb$}).$
Suppose that $X = f_{i_l}\cdots f_{i_2} f_{i_1}$ for $i_k \in I$.
Let $j_k := \sqrt{i_k+\half}\sqrt{i_k-\half}$ and consider the supermodule
$$
P := \sF_{j_l} \cdots \sF_{j_2} \sF_{j_1} \,M(\ba).
$$
Since $\ba$ is dominant and typical, $M(\ba)$ is projective.
Since each $\sF_j$ sends projectives to projectives (being left
adjoint to an exact functor), 
we deduce that $P$ is projective.
Since the combinatorics of 
(\ref{cg}) matches that of Theorem~\ref{maintfthm2},
we can
reinterpret
Lemma~\ref{construction2} as saying
that $P$
has a Verma flag with 
one section evenly isomorphic to $M(\bb)$
and all other sections evenly isomorphic to $M(\bc)$'s for $\bc \succ
\bb$.
In fact, 
the unique section isomorphic to $M(\bb)$ appears at the top of this Verma flag, thanks
the order of the sections arising from Theorem~\ref{maintfthm2}(1).
Hence, $P$ has a summand evenly isomorphic to
$P(\bb)$, and we are done as $\sO^\Delta$ is closed under passing to
summands.
\end{proof}

\begin{corollary}\label{bruhat}
For $\bc \in \B$, we have that $[M(\bc):L(\bc)] = 1$. All other composition factors of $M(\bc)$
are evenly isomorphic to $L(\ba)$'s for $\ba \prec \bc$.
\end{corollary}

\begin{proof}
This follows from Theorem~\ref{mainsplitthm2} using {\em BGG reciprocity}:
for $\ba, \bc \in \B$ and $p \in \Z/2$, we have that
$$
[M(\bc):\Pi^p L(\ba)]
= [M(\bc)^\star:\Pi^p L(\ba)]
= \dim \Hom_{\sO}(P(\ba), M(\bc)^\star)_p
=
(P(\ba):\Pi^p M(\bc)).
$$
\end{proof}

\begin{corollary}\label{split}
For any $\bb \in \B$, every irreducible subquotient of the
indecomposable projective $P(\bb)$ is evenly isomorphic to $L(\ba)$
for $\ba \in \B$ with $|\WT(\ba)| = |\WT(\bb)|$.
\end{corollary}

\begin{proof}
By Theorem~\ref{mainsplitthm2} and Corollary~\ref{bruhat}, the composition factors of $P(\bb)$ are 
$L(\ba)$'s for $\ba \in \B$ such that $\ba \preceq \bc
\succeq \bb$ for some $\bc$.
This implies that $|\WT(\ba)| = |\WT(\bb)|$.
\end{proof}

\subsection{The main categorification theorem}
Recall that $I = \N$.
The monoidal category in the following definition is 
one of the categories introduced
by Khovanov and Lauda
\cite{KL1, KL2} and Rouquier \cite{Rou}, for the graph arising from the Dynkin
diagram of
$\mathfrak{sp}_{2\infty}$
and the matrix of parameters $(q_{i,j}(u,v))_{i,j \in I}$ defined from
\begin{equation}\label{parameters}
q_{i,j}(u,v) :=
\left\{
\begin{array}{ll}
0&\text{if $i=j$,}\\
1&\text{if $|i-j|>1$,}\\
u^2-v &\text{if $i=1$ and $j=0$,}\\
v^2-u&\text{if $i=0$ and $j=1$,}\\
(i-j)u +(j-i)v
&\text{otherwise.}
\end{array}
\right.
\end{equation}

\begin{definition}\label{QH1}
The \textit{quiver Hecke category} $\QH$
of type $\mathfrak{sp}_{2\infty}$
is the strict $\K$-linear monoidal category generated by
objects $I$
and 
morphisms
$\mathord{
\begin{tikzpicture}[baseline = -2]
	\draw[-,thick,darkred] (0.08,-.15) to (0.08,.3);
      \node at (0.08,0.05) {$\color{darkred}\bullet$};
   \node at (0.08,-.25) {$\scriptstyle{i}$};
\end{tikzpicture}
}:i \rightarrow i$ and
$\mathord{
\begin{tikzpicture}[baseline = -2]
	\draw[-,thick,darkred] (0.18,-.15) to (-0.18,.3);
	\draw[-,thick,darkred] (-0.18,-.15) to (0.18,.3);
   \node at (-0.18,-.25) {$\scriptstyle{i_2}$};
   \node at (0.18,-.25) {$\scriptstyle{i_1}$};
\end{tikzpicture}
}:i_2 \otimes i_1 \rightarrow i_1 \otimes i_2$
subject to the following relations:
\begin{align*}
%y and tau
\mathord{
\begin{tikzpicture}[baseline = 2]
	\draw[-,thick,darkred] (0.25,.6) to (-0.25,-.2);
	\draw[-,thick,darkred] (0.25,-.2) to (-0.25,.6);
  \node at (-0.25,-.28) {$\scriptstyle{i_2}$};
   \node at (0.25,-.28) {$\scriptstyle{i_1}$};
      \node at (-0.14,0.42) {$\color{darkred}\bullet$};
\end{tikzpicture}
}
-
\mathord{
\begin{tikzpicture}[baseline =2]
	\draw[-,thick,darkred] (0.25,.6) to (-0.25,-.2);
	\draw[-,thick,darkred] (0.25,-.2) to (-0.25,.6);
  \node at (-0.25,-.28) {$\scriptstyle{i_2}$};
   \node at (0.25,-.28) {$\scriptstyle{i_1}$};
      \node at (0.14,-0.02) {$\color{darkred}\bullet$};
\end{tikzpicture}
}
&=
\mathord{
\begin{tikzpicture}[baseline = 2]
	\draw[-,thick,darkred] (0.25,.6) to (-0.25,-.2);
	\draw[-,thick,darkred] (0.25,-.2) to (-0.25,.6);
  \node at (-0.25,-.28) {$\scriptstyle{i_2}$};
   \node at (0.25,-.28) {$\scriptstyle{i_1}$};
      \node at (-0.13,-0.02) {$\color{darkred}\bullet$};
\end{tikzpicture}
}
-\mathord{
\begin{tikzpicture}[baseline = 2]
	\draw[-,thick,darkred] (0.25,.6) to (-0.25,-.2);
	\draw[-,thick,darkred] (0.25,-.2) to (-0.25,.6);
  \node at (-0.25,-.28) {$\scriptstyle{i_2}$};
   \node at (0.25,-.28) {$\scriptstyle{i_1}$};
      \node at (0.14,0.42) {$\color{darkred}\bullet$};
\end{tikzpicture}
}
=
\left\{
\begin{array}{ll}
\mathord{
\begin{tikzpicture}[baseline = -1]
 	\draw[-,thick,darkred] (0.08,-.3) to (0.08,.4);
	\draw[-,thick,darkred] (-0.28,-.3) to (-0.28,.4);
   \node at (-0.28,-.4) {$\scriptstyle{i_2}$};
   \node at (0.08,-.4) {$\scriptstyle{i_1}$};  
   \end{tikzpicture} 
   }
   & \text{if $i_1 = i_2$,}\\
	\:\:\:0 & \text{if $i_1 \neq i_2$;}\\
\end{array}
\right. \end{align*}\begin{align*}
%tau ^2 
\mathord{
\begin{tikzpicture}[baseline = 8]
	\draw[-,thick,darkred] (0.28,.4) to[out=90,in=-90] (-0.28,1.1);
	\draw[-,thick,darkred] (-0.28,.4) to[out=90,in=-90] (0.28,1.1);
	\draw[-,thick,darkred] (0.28,-.3) to[out=90,in=-90] (-0.28,.4);
	\draw[-,thick,darkred] (-0.28,-.3) to[out=90,in=-90] (0.28,.4);
  \node at (-0.28,-.4) {$\scriptstyle{i_2}$};
  \node at (0.28,-.4) {$\scriptstyle{i_1}$};
\end{tikzpicture}
}
&=
\left\{
\begin{array}{ll}
\:\:\:0&\text{if $i_1 = i_2$,}\\
\mathord{
\begin{tikzpicture}[baseline = 0]
	\draw[-,thick,darkred] (0.08,-.3) to (0.08,.4);
	\draw[-,thick,darkred] (-0.28,-.3) to (-0.28,.4);
   \node at (-0.28,-.4) {$\scriptstyle{i_2}$};
   \node at (0.08,-.4) {$\scriptstyle{i_1}$};
\end{tikzpicture}
}&\text{if $|i_1 - i_2| > 1$,}\\
\mathord{
\begin{tikzpicture}[baseline = 0]
	\draw[-,thick,darkred] (0.08,-.3) to (0.08,.4);
	\draw[-,thick,darkred] (-0.28,-.3) to (-0.28,.4);
   \node at (-0.28,-.4) {$\scriptstyle{i_2}$};
   \node at (0.08,-.4) {$\scriptstyle{i_1}$};
      \node at (-0.28,-0.05) {$\color{darkred}\bullet$};
      \node at (-0.28,0.15) {$\color{darkred}\bullet$};
\end{tikzpicture}
}
-\mathord{
\begin{tikzpicture}[baseline = 0]
	\draw[-,thick,darkred] (0.08,-.3) to (0.08,.4);
	\draw[-,thick,darkred] (-0.28,-.3) to (-0.28,.4);
   \node at (-0.28,-.4) {$\scriptstyle{i_2}$};
   \node at (0.08,-.4) {$\scriptstyle{i_1}$};
     \node at (0.08,0.05) {$\color{darkred}\bullet$};
\end{tikzpicture}
}&\text{if $i_1=0$ and $i_2 = 1$,}\\
\mathord{
\begin{tikzpicture}[baseline = 0]
	\draw[-,thick,darkred] (0.08,-.3) to (0.08,.4);
	\draw[-,thick,darkred] (-0.28,-.3) to (-0.28,.4);
   \node at (-0.28,-.4) {$\scriptstyle{i_2}$};
   \node at (0.08,-.4) {$\scriptstyle{i_1}$};
      \node at (0.08,-0.05) {$\color{darkred}\bullet$};
      \node at (0.08,0.15) {$\color{darkred}\bullet$};
\end{tikzpicture}
}
-\mathord{
\begin{tikzpicture}[baseline = 0]
	\draw[-,thick,darkred] (0.08,-.3) to (0.08,.4);
	\draw[-,thick,darkred] (-0.28,-.3) to (-0.28,.4);
   \node at (-0.28,-.4) {$\scriptstyle{i_2}$};
   \node at (0.08,-.4) {$\scriptstyle{i_1}$};
     \node at (-0.28,0.05) {$\color{darkred}\bullet$};
\end{tikzpicture}
}
&\text{if $i_1=1$ and $i_2=0$,}\\
 (i_1-i_2)
\mathord{
\begin{tikzpicture}[baseline = 0]
	\draw[-,thick,darkred] (0.08,-.3) to (0.08,.4);
	\draw[-,thick,darkred] (-0.28,-.3) to (-0.28,.4);
   \node at (-0.28,-.4) {$\scriptstyle{i_2}$};
   \node at (0.08,-.4) {$\scriptstyle{i_1}$};
     \node at (0.08,0.05) {$\color{darkred}\bullet$};
\end{tikzpicture}
}
+ (i_2-i_1)
\mathord{
\begin{tikzpicture}[baseline = 0]
	\draw[-,thick,darkred] (0.08,-.3) to (0.08,.4);
	\draw[-,thick,darkred] (-0.28,-.3) to (-0.28,.4);
   \node at (-0.28,-.4) {$\scriptstyle{i_2}$};
   \node at (0.08,-.4) {$\scriptstyle{i_1}$};
      \node at (-0.28,0.05) {$\color{darkred}\bullet$};
\end{tikzpicture}
}
&\text{otherwise;}\\
\end{array}
\right. 
\end{align*}\begin{align*}
%braid relation
\mathord{
\begin{tikzpicture}[baseline = 2]
	\draw[-,thick,darkred] (0.45,.8) to (-0.45,-.4);
	\draw[-,thick,darkred] (0.45,-.4) to (-0.45,.8);
        \draw[-,thick,darkred] (0,-.4) to[out=90,in=-90] (-.45,0.2);
        \draw[-,thick,darkred] (-0.45,0.2) to[out=90,in=-90] (0,0.8);
   \node at (-0.45,-.5) {$\scriptstyle{i_3}$};
   \node at (0,-.5) {$\scriptstyle{i_2}$};
  \node at (0.45,-.5) {$\scriptstyle{i_1}$};
\end{tikzpicture}
}
\!\!-
\!\!\!
\mathord{
\begin{tikzpicture}[baseline = 2]
	\draw[-,thick,darkred] (0.45,.8) to (-0.45,-.4);
	\draw[-,thick,darkred] (0.45,-.4) to (-0.45,.8);
        \draw[-,thick,darkred] (0,-.4) to[out=90,in=-90] (.45,0.2);
        \draw[-,thick,darkred] (0.45,0.2) to[out=90,in=-90] (0,0.8);
   \node at (-0.45,-.5) {$\scriptstyle{i_3}$};
   \node at (0,-.5) {$\scriptstyle{i_2}$};
  \node at (0.45,-.5) {$\scriptstyle{i_1}$};
\end{tikzpicture}
}
&=
\left\{
\begin{array}{ll}
\mathord{
\begin{tikzpicture}[baseline = -1]
	\draw[-,thick,darkred] (0.44,-.3) to (0.44,.4);
	\draw[-,thick,darkred] (0.08,-.3) to (0.08,.4);
	\draw[-,thick,darkred] (-0.28,-.3) to (-0.28,.4);
   \node at (-0.28,-.4) {$\scriptstyle{i_3}$};
   \node at (0.08,-.4) {$\scriptstyle{i_2}$};
   \node at (0.44,-.4) {$\scriptstyle{i_1}$};
      \node at (0.44,0.05) {$\color{darkred}\bullet$};
\end{tikzpicture}
}
+
\mathord{
\begin{tikzpicture}[baseline = -1]
	\draw[-,thick,darkred] (0.44,-.3) to (0.44,.4);
	\draw[-,thick,darkred] (0.08,-.3) to (0.08,.4);
	\draw[-,thick,darkred] (-0.28,-.3) to (-0.28,.4);
   \node at (-0.28,-.4) {$\scriptstyle{i_3}$};
   \node at (0.08,-.4) {$\scriptstyle{i_2}$};
   \node at (0.44,-.4) {$\scriptstyle{i_1}$};
      \node at (-0.28,0.05) {$\color{darkred}\bullet$};
\end{tikzpicture}
}
&
\text{if $i_1 = i_3=1$ and $i_2=0$,}\\
(i_1 - i_2) 
\mathord{
\begin{tikzpicture}[baseline = -1]
	\draw[-,thick,darkred] (0.44,-.3) to (0.44,.4);
	\draw[-,thick,darkred] (0.08,-.3) to (0.08,.4);
	\draw[-,thick,darkred] (-0.28,-.3) to (-0.28,.4);
   \node at (-0.28,-.4) {$\scriptstyle{i_3}$};
   \node at (0.08,-.4) {$\scriptstyle{i_2}$};
   \node at (0.44,-.4) {$\scriptstyle{i_1}$};
\end{tikzpicture}
}
&
\text{if $i_1 = i_3$, $|i_1 - i_2|= 1$ and $i_2 \neq 0$,}\\
\:\:\:0 &\text{otherwise.}
\end{array}
\right.
\end{align*}
Although we will not make use of it in this section, we note
that $\QH$ can be enriched with a $\Z$-grading by setting $\operatorname{deg}\left(\mathord{
\begin{tikzpicture}[baseline = -2]
	\draw[-,thick,darkred] (0.08,-.15) to (0.08,.3);
      \node at (0.08,0.05) {$\color{darkred}\bullet$};
   \node at (0.08,-.25) {$\scriptstyle{i}$};
\end{tikzpicture}
}\right) := (\alpha_i,\alpha_i)$
and $\operatorname{deg}\left(\mathord{
\begin{tikzpicture}[baseline = -2]
	\draw[-,thick,darkred] (0.18,-.15) to (-0.18,.3);
	\draw[-,thick,darkred] (-0.18,-.15) to (0.18,.3);
   \node at (-0.18,-.25) {$\scriptstyle{i_2}$};
   \node at (0.18,-.25) {$\scriptstyle{i_1}$};
\end{tikzpicture}
}\right) := -(\alpha_{i_1}, \alpha_{i_2})$.
\end{definition}

Our final definition is the analog for $\mathfrak{sp}_{2\infty}$ 
of \cite[Definition 2.10]{BLW}, which reformulated
\cite[Definiton 3.2]{LW} for tensor products of minuscule
representations of $\mathfrak{sl}_{\infty}$.

\begin{definition}\label{tpcdef}
A {\em tensor product categorification}
(TPC for short) of
$V^{\otimes n}$
is the following data:
\begin{itemize}
\item
a highest weight category $\mathcal C$ 
with 
standard objects $\{\Delta(\bb)\:|\:\bb \in \B\}$
indexed by the set $\B$
ordered according to the Bruhat order $\preceq$;
\item
adjoint pairs $(F_i, E_i)$ of endofunctors of $\mathcal C$ for each $i
\in I$;
\item
a strict monoidal functor $\Phi:\QH \rightarrow
\mathcal{E}nd(\C)$ with $\Phi(i) = F_i$ for each $i \in I$.
\end{itemize}
We impose the following additional axioms for all $i
\in I$ and $\bb \in \B$:
\begin{itemize}
\item
$E_i$ is isomorphic to a left adjoint
of $F_i$;
\item
$F_i \Delta(\bb)$ 
has a filtration with sections 
$\{\Delta(\bb+\bd_t)\:|\: 1 \leq t \leq n,
\isig_t(\bb) = \mathtt{f}\}$;
\item
$E_i \Delta(\bb)$ 
has a filtration with sections 
$\{\Delta(\bb-\bd_t)\:|\:1 \leq t \leq n, \isig_t(\bb) =
\mathtt{e}\}$;
\item
the natural transformation
$\Phi\Big(
\mathord{
\begin{tikzpicture}[baseline = -2]
	\draw[-,thick,darkred] (0.08,-.15) to (0.08,.3);
      \node at (0.08,0.05) {$\color{darkred}\bullet$};
   \node at (0.08,-.25) {$\scriptstyle{i}$};
\end{tikzpicture}
}
\Big)$ is {locally nilpotent}.
\end{itemize}
\end{definition}

Now we let $\O$ (resp. $\Pi \O$) be the subcategory of $\sO$ 
consisting of the supermodules all of whose composition factors are
evenly (resp. oddly) isomorphic to
$L(\bb)$'s for $\bb \in \B$.
All morphisms between objects of $\O$ are purely even, so we may as
well  forget the $\Z/2$-grading and view $\O$ simply as a
$\K$-linear category.

Our main theorem is as follows.

\begin{theorem}\label{typec}
We have that $\sO = \O \oplus \Pi \O$, i.e. the supercategory $\sO$ splits.
Moreover, the $\K$-linear category $\O$
admits all of the additional structure needed to make it into a TPC of $V^{\otimes n}$.
\end{theorem}

\begin{proof}
The fact that $\sO = \O \oplus \Pi \O$ follows from
Corollary~\ref{split}; cf. the proof of \cite[Theorem 5.1]{BD}.
To make $\O$ into a TPC, we need to
introduce the additional data then check the axioms from Definition~\ref{tpcdef}.

It is clear that $\O$ is a Schurian category in the sense of
Definition~\ref{schurian} with irreducible objects
$\{L(\bb)\:|\:\bb \in \B\}$.
Since $P(\bb)$ belongs to $\O$, it is the projective cover of $L(\bb)$
in $\O$. Theorem~\ref{mainsplitthm2} and Corollary~\ref{bruhat} then give the necessary
technical ingredients needed to check that $\O$ is a highest weight
category with the required weight poset; cf. the proof of
\cite[Theorem 5.4]{BD}.
Its standard objects $\{\Delta(\bb)\:|\:\bb\in\B\}$ are the Verma supermodules $\{M(\bb)\:|\:\bb\in\B\}$.

To define $F_i$ and $E_i$, take $i \in I$, and let 
$j := \sqrt{i+\half}\sqrt{i-\half}$.
Theorem~\ref{maintfthm2}
shows that $\sF_j \,M(\bb)$ and $\sE_j \,M(\bb)$ are objects of $\O$.
Hence, by exactness, the functors $\sF_j$ and $\sE_j$ send arbitrary objects from
$\O$ to objects of $\O$. So we obtain the required endofunctors by
setting
$$
F_i := \sF_j|_{\O}:\O\rightarrow\O, 
\qquad
E_i := \sE_j|_{\O}:\O\rightarrow\O.
$$
The adjunctions $(\sF_j, \sE_j)$ and $(\sE_j, \sF_j)$ discussed
earlier give adjunctions $(F_i, E_i)$ and $(E_i, F_i)$ too. 
Also $F_i \,M(\bb)$ and $E_i \,M(\bb)$ have the required Verma
filtrations thanks to Theorem~\ref{maintfthm2}. 

It remains to define $\Phi:\QH \rightarrow
\mathcal{E}nd(\O)$.
Since $\QH$ is defined by generators and relations, we can
do this simply by declaring that $\Phi(i) := F_i$ for each $i$, then 
specifying natural transformations
$\Phi\Big(
\mathord{
\begin{tikzpicture}[baseline = -2]
	\draw[-,thick,darkred] (0.08,-.15) to (0.08,.3);
      \node at (0.08,0.05) {$\color{darkred}\bullet$};
   \node at (0.08,-.25) {$\scriptstyle{i}$};
\end{tikzpicture}
}
\Big):F_i \Rightarrow F_i$
and
$\Phi\Big(\mathord{
\begin{tikzpicture}[baseline = -2]
	\draw[-,thick,darkred] (0.18,-.15) to (-0.18,.3);
	\draw[-,thick,darkred] (-0.18,-.15) to (0.18,.3);
   \node at (-0.18,-.25) {$\scriptstyle{i_2}$};
   \node at (0.18,-.25) {$\scriptstyle{i_1}$};
\end{tikzpicture}
}\Big):F_{i_2} F_{i_1} \Rightarrow F_{i_1}F_{i_2}$ satisfying the quiver
Hecke relations from Definition~\ref{QH1}.
The explicit formulae for these natural transformations are recorded 
in the next two paragraphs.
They were derived like in the proof of \cite[Theorem 6.2]{BD}
by starting from
the supernatural transformations from Theorem~\ref{hkst},
which satisfy the affine Hecke-Clifford relations of
Definition~\ref{AHC}, then using the remarkable isomorphism from
\cite[Theorem 5.4]{KKT}
to combine these into supernatural transformations satisfying 
the quiver Hecke-Clifford relations of \cite[Definition 3.5]{KKT}.
When $i=0$, the number $j =
\sqrt{i+\half}\sqrt{i-\half}$
satisfies $j^2 + \frac{1}{4} = 0$. Hence, we are in the situation of
\cite[$\S$5.2(i)(c)]{KKT} and the appropriate Dynkin diagram is
of type $\mathfrak{sp}_{2\infty}$, unlike in \cite{BD} where it was of
type $\mathfrak{sl}_\infty$. This is really the only difference
compared to the proof of \cite[Theorem 6.2]{BD}, so we omit any
further explanations.

Here we give the explicit formula for
$\Phi\Big(
\mathord{
\begin{tikzpicture}[baseline = -2]
	\draw[-,thick,darkred] (0.08,-.15) to (0.08,.3);
      \node at (0.08,0.05) {$\color{darkred}\bullet$};
   \node at (0.08,-.25) {$\scriptstyle{i}$};
\end{tikzpicture}
}
\Big)$.
Let $j := \sqrt{i+\half}\sqrt{i-\half}$, then define
$y_1 \in (x_1-j)\K[[x_1-j]]$ to be $x_1^2+\frac{1}{4}$ if $i=0$,
or 
the unique power series in $(x_1-j)\K[[x_1-j]]$ such that
$(y_1+i)^2 = x_1^2+\frac{1}{4}$ if $i \neq 0$.
Recalling (\ref{completed}), this gives us an element $y_1 1_{j} \in \widehat{AHC}_1$. Applying the
homomorphism $\widehat{\Psi}_1$ from (\ref{tea}), we obtain from this
an even supernatural transformation $\widehat{\Psi}_1(y_1 1_j):\sF_{j}
\rightarrow \sF_j$. 
Since $F_i$ is the restriction of $\sF_j$, this gives us the required
natural transformation
$\Phi\Big(
\mathord{
\begin{tikzpicture}[baseline = -2]
	\draw[-,thick,darkred] (0.08,-.15) to (0.08,.3);
      \node at (0.08,0.05) {$\color{darkred}\bullet$};
   \node at (0.08,-.25) {$\scriptstyle{i}$};
\end{tikzpicture}
}
\Big)$. It is locally nilpotent because $(x-j)$ acts locally
nilpotently on $\sF_j$ by the definition of $\sF_j$.

Finally, we give the formula for
$\Phi\Big(\mathord{
\begin{tikzpicture}[baseline = -2]
	\draw[-,thick,darkred] (0.18,-.15) to (-0.18,.3);
	\draw[-,thick,darkred] (-0.18,-.15) to (0.18,.3);
   \node at (-0.18,-.25) {$\scriptstyle{i_2}$};
   \node at (0.18,-.25) {$\scriptstyle{i_1}$};
\end{tikzpicture}\Big)
}$.
For this, we work in $\widehat{AHC}_2$.
For $r = 1,2$, let $j_r := \sqrt{i_r+\half}\sqrt{i_r-\half}$, then
define $y_r \in (x_r-j_r)\K[[x_r-j_r]$
to be $x_r^2+\frac{1}{4}$ if $i_r=0$,
or 
the unique power series
such that
$(y_r+i_r)^2 = x_r^2+\frac{1}{4}$ if $i_r \neq 0$ (like in the
previous paragraph).
Let 
$$
p:=\frac{(x_1^2-x_2^2)^2}{2(x_1^2+x_2^2)-(x_1^2-x_2^2)^2},
$$
which is an element of $\K[[x_1-j_1,x_2-j_2]]$ unless $|i_1- i_2|=1$ (when it should be viewed as an element of the fraction field).
Then, recalling (\ref{parameters}), we define $g \in
\K[[x_1-j_1,x_2-j_2]]$ from
$$
g :=
\left\{
\begin{array}{ll}
-1&\text{if $i_1 < i_2$,}\\
\sqrt{p}/(y_1-y_2)&\text{if $i_1= i_2$,}\\
 p\:q_{i_2,i_1}\!(y_2,y_1)&\text{if $i_1 > i_2$,}
\end{array}\right.
$$
choosing the square root when $i_1=i_2$
so that $g -\frac{x_1-x_2}{y_1-y_2} \in
(x_1-x_2)\K[[x_1-j_1,x_2-j_2]]$.
Using (\ref{opaque1})--(\ref{opaque2}), one can check that $$
t_1 g 1_{j_2 j_1} +
\left(\frac{g}{x_1-x_2}-\frac{\delta_{i_1,i_2}}{y_1-y_2}\right) 1_{j_2 j_1}
+ c_1 c_2 \frac{g}{x_1+x_2}1_{j_2 j_1} \in 
1_{j_1 j_2} \widehat{AHC}_2 1_{j_2 j_1}.
$$
Applying $\widehat{\Psi}_2$, we obtain an even 
supernatural transformation $\sF_{j_2} \sF_{j_1}\Rightarrow \sF_{j_1}
\sF_{j_2}$, hence, the desired natural transformation $F_{i_2} F_{i_1}
\Rightarrow F_{i_1} F_{i_2}$.
\end{proof}

\section{Orthodox basis}

In this section, we prove the first Cheng-Kwon-Wang conjecture
\cite[Conjecture 5.12]{CKW}.
Throughout the section, $I$ will denote the set $\N$ that indexes the
simple roots of $\mathfrak{sp}_{2\infty}$, and $\B = \Z^n$ as always.
For $k \geq 1$, we'll write $I_k$ for the set $\{0,1,\dots,k-1\}$ that
indexes the simple roots of the subalgebra $\mathfrak{sp}_{2k} <
\mathfrak{sp}_{2\infty}$, and define $\B_k$ as in (\ref{Bsigma}).

\subsection{\boldmath Truncation from $\mathfrak{sp}_{2\infty}$ to $\mathfrak{sp}_{2k}$}\label{roger1}
Fix $k \geq 1$.
The quiver Hecke category of type $\mathfrak{sp}_{2k}$ is the
full subcategory $\QH_k$ of $\QH$ whose objects are monoidally generated
by $I_k \subset I$.
There is a notion of a {\em tensor product categorification
 of $V_k^{\otimes n}$}. This is defined in exactly the same way as
Definition~\ref{tpcdef},
replacing $\mathfrak{sp}_{2\infty}, V, \B, I$ and $\QH$ with
$\mathfrak{sp}_{2k}, V_k, \B_k, I_k$ and $\QH_k$, respectively.
In this subsection, we are going to explain how to construct such a
structure
out of a TPC of $V^{\otimes n}$ 
by passing to a certain subquotient.
The approach is similar to that of \cite[$\S$2.8]{BLW}.

Recall (\ref{N2}). Let $\B_{\leq k}$ denote the set of all $\bb \in \B$
such that
$N_{[1,s]}(\bb,k) \leq 0$
for $s=1,\dots,n-1$ and $N_{[1,n]}(\bb,k) = 0$.
Let $\B_{< k}$ be the set of all $\bb
\in \B_{\leq k}$ such that $N_{[1,s]}(\bb,k) < 0$
for at least one $s$.
Lemma~\ref{spequiv} implies that these are both ideals (lower sets) in
the poset $\B$.
Observe moreover that $\B_k$ is the set difference
$\B_{\leq k} \setminus \B_{< k}$.

Now let $\C$ be any TPC
of $V^{\otimes n}$.
Let $\C_{\leq k}$ be the Serre subcategory of $\C$ generated
by the irreducible supermodules $\{L(\bb)\:|\:\bb \in \B_{\leq k}\}$,
and define $\C_{< k}$ similarly using $\B_{< k}$. 
As $\B_{\leq k}$ and $\B_{< k}$ are ideals, 
we are in the same general situation
as discussed in \cite[$\S$2.5]{BLW}.
Hence,
$\C_{\leq k}$ and
$\C_{< k}$ get
induced highest weight structures, as does the Serre
quotient $\C_k := \C_{\leq k} / \C_{< k}$.
Its weight poset is $(\B_k, \preceq)$.

\begin{theorem}\label{fk}
The subquotient $\C_k$ of $\C$
admits the structure of a TPC of $V_k^{\otimes n}$.
\end{theorem}

\begin{proof}
We must check all of the properties from the $\mathfrak{sp}_{2k}$
version of Definition~\ref{tpcdef}.
We've already explained that $\C_k$ is a highest weight category with
the appropriate weight poset.
Next, we show that the endofunctors $E_i, F_i$ for $i \in I_k$ leave
both $\C_{\leq k}$ and $\C_{< k}$ invariant. 
As in the proof of \cite[Lemma 2.18]{BLW}, 
we just need to verify this on standard objects, when it follows using
the observation that $$
N_{[1,s]}(\bb\pm \bd_r,k) =
N_{[1,s]}(\bb,k)
$$
for all $\bb \in \B$ and $r,s=1,\dots,n$ such that $\isig_r(\bb) \in
\{\mathtt{e},\mathtt{f}\}$ for $i \in I_k$.
Hence, $E_i, F_i$ 
induce biadjoint endofunctors of $\C_k$ for each $i \in I_k$.
All of the other required structure comes immediately from the definitions.
\end{proof}

\subsection{Proof of the first Cheng-Kwon-Wang conjecture}
Our definition of a TPC
of $V_k^{\otimes n}$ is a simplified version of the more general
notion of TPC from \cite[Definition 3.2]{LW}.
The simplification is possible because $V_k$ is a minuscule highest
weight representation for $\mathfrak{sp}_{2k}$.
The equivalence of our definition with the Losev-Webster definition
may be verified by a similar argument to the one explained in
\cite[Remark 2.11]{BLW}.
Hence, we obtain the following as a special case of the
uniqueness theorem for TPCs that is the
main result of \cite{LW}; we refer to \cite[Definition 4.7]{BD1} for
the definition of strongly equivariant equivalence being used here.

\begin{theorem}[Losev--Webster]\label{uniqueness}
All TPCs of
$V_k^{\otimes n}$ are strongly equivariantly equivalent via
equivalences which preserve the labelling of irreducible objects.
\end{theorem}

If we apply the construction from the previous subsection to the
category $\O$ of Theorem~\ref{typec}, we obtain a subquotient $\O_k :=
\O_{\leq k} / \O_{< k}$ of $\O$
which is a TPC of $V_k^{\otimes
  n}$.
Let $\Web_k$ denote Webster's tensor product algebra
associated to the $n$-fold
tensor product
of the natural representation of $\mathfrak{sp}_{2k}$, that is, 
the algebra $T^{\bomega_k}$ from \cite[$\S$4]{Web}
associated to the $n$-tuple of dominant weights
$\bomega_k := (-\eps_{k-1},\dots,-\eps_{k-1})$ for $\mathfrak{sp}_{2k}$.
%Note $|\bomega_k|$ is the highest weight $\omega_k$ of $V^{\otimes
%  n}_k$ from (\ref{omk}).
Webster's general theory from \cite{Web}
shows that the category $\Web_k\text{-}\mod$ of finite dimensional modules
over this algebra also has the structure of a TPC of $V^{\otimes n}_k$;
see also \cite[Theorem 3.12]{LW}.
Hence, applying Theorem~\ref{uniqueness}, we obtain the following:

\begin{corollary}\label{here}
The category $\O_k$ is equivalent to $\Web_k\text{-}\mod$
via an equivalence which preserves the labelling of irreducible objects.
\end{corollary}

In particular, this means that the combinatorics of decomposition
numbers in the category $\O$ is the same as that of Webster's tensor
product algebras. 
More precisely, given any $\ba,\bb \in \B$, we pick $k$ large enough
so that $\ba, \bb$ both belong to $\B_k$.
Then, 
Corollary~\ref{here}
implies that
\begin{equation}\label{army}
[M(\ba):L(\bb)]
= [M_k(\ba):L_k(\bb)],
\end{equation}
where $M_k(\ba)$ denotes the 
standard
$\Web_k$-module associated to $\ba \in \B_k$ as constructed in
\cite[$\S$5]{Web}, and $L_k(\ba)$ is its unique irreducible quotient.
Indeed, for $\ba \in \B_k$, 
the canonical images of the standard objects $M(\ba)$ and their
irreducible quotients
$L(\ba)$ in the quotient category $\O_k$
map under the equivalence from Corollary~\ref{here} to copies of $M_k(\ba)$ and
$L_k(\ba)$, respectively.
Then (\ref{army}) follows just like
in the proof of \cite[Theorem 2.21]{BLW}, 

We can reformulate the assertions made in the previous paragraph in
terms of Webster's {orthodox basis}, as follows.
Let $P_k(\ba)$ be the projective cover of $L_k(\ba)$
in $\Web_k\text{-}\mod$. As
$\Web_k\text{-}\mod$ is a TPC, there
is a vector space isomorphism
$$
\iota_k:\CC \otimes_{\Z} K_0(\Web_k\text{-}\mod) \stackrel{\sim}{\rightarrow} V_k^{\otimes n},
\qquad
[M_k(\ba)] \mapsto v_\ba.
$$
By the definition from \cite[$\S$7]{Web1}, Webster's {\em orthodox basis} of $V_k^{\otimes n}$
(specialized at $q=1$) is the basis
$\left\{\iota_k([P_k(\bb)])\:\big|\:\bb \in \B_k\right\}.$
Analogously, we can consider the isomorphism
$$
\iota:\CC \otimes_{\Z} K_0(\mathcal O^\Delta) \stackrel{\sim}{\rightarrow} V^{\otimes n},
\qquad
[M(\ba)] \mapsto v_\ba.
$$
The following defines the {\em orthodox basis} of
$V^{\otimes n}$ (specialized at $q=1$).

\begin{theorem}\label{lastyear}
The space $V^{\otimes n}$ has a unique topological basis
$\{o_\bb\:|\:b \in \B\}$ such that $\pr_k o_\bb = \iota_k([P_k(\bb)])$
for each $k \geq 1$ and $\bb \in \B_k$.
Moreover, we have that $o_\bb = \iota([P(\bb)])$ for any $\bb \in \B$.
\end{theorem}

\begin{proof}
Let $o_\bb := \iota([P(\bb)])$.
By BGG reciprocity in the highest weight categories $\O$ and
$A_k\text{-}\mod$, respectively, we have that $[P(\bb)] = \sum_{\ba \in \B}
[M(\ba):L(\bb)] [M(\ba)]$ and
$[P_k(\bb)] = \sum_{\ba \in \B_k}
[M_k(\ba):L_k(\bb)] [M_k(\ba)]$.
Hence, for $\bb \in \B_k$, we have that
$$
\pr_k o_\bb 
= \sum_{\ba \in \B_k} [M(\ba):L(\bb)] v_\ba
= \sum_{\ba \in \B_k} [M_k(\ba):L_k(\bb)] v_\ba=
\iota_k([P_k(\bb)]),
$$
using (\ref{army}) for the middle equality.
\end{proof}

This establishes the truth of \cite[Conjecture 5.12]{CKW}. 
Actually, Cheng, Kwon and Wang formulated their conjecture 
in terms of tilting modules instead of projective modules, i.e. they
work in the
Ringel dual setting. The equivalence of our Theorem~\ref{lastyear} with their
conjecture follows by \cite[(7.12)]{Btilt}.

\begin{remark}\label{or}
Webster's algebra $A_k$ admits a natural $\Z$-grading.
Hence, one can consider the category $A_k\text{-}\grmod$ of
finite-dimensional {\em graded} $A_k$-modules. The endofunctors $E_i$
and $F_i$ also admit graded lifts, making $A_k\text{-}\grmod$ into a
{\em $U_q \mathfrak{sp}_{2k}$-tensor product categorification} of
$\dot V_k^{\otimes n}$. We
refer the reader to
\cite[Definition 5.9]{BLW} for a related definition which is easily adapted to the present
situation; this depends on the grading on $\QH_k$ noted at the end of
Definition~\ref{QH1}.
The Grothendieck group $K_0(A\text{-}\grmod)$ is a
$\Z[q,q^{-1}]$-module with $q$ acting as the upward grading shift
functor. Also the standard modules $M_k(\ba)$ admit graded lifts $\dot
M_k(\ba)$, such that there is a $\Q(q)$-vector space isomorphism
$$
\dot\iota_k:\Q(q) \otimes_{\Z[q,q^{-1}]} K_0(A_k\text{-}\grmod)
\stackrel{\sim}{\rightarrow} \dot V_k^{\otimes n},
\qquad
[\dot M_k(\ba)] \mapsto \dot v_\ba.
$$
Webster's {\em orthodox basis} of $\dot V_k^{\otimes n}$
is the basis $\big\{\dot \iota_k([\dot P_k(\bb)])\:\big|\:\bb \in \B_k\big\}$, where $\dot P_k(\bb)$
is the projective cover of $\dot M_k(\bb)$ in $A_k\text{-}\grmod$.
Using the graded analog of 
Theorem~\ref{uniqueness}, one can show that the coefficients of this basis stabilize as $k
\rightarrow \infty$, hence, there is a unique topological basis
$\{\dot o_\bb \:|\:\bb \in \B\}$ for $\dot V^{\otimes n}$
such that $\pr_k \dot o_\bb = \dot \iota_k([\dot P_k(\bb)])$ for all
$k \geq 1$ and $\bb \in \B_k$. This is the $q$-analog of the basis in Theorem~\ref{lastyear}.
\end{remark}

\begin{remark}\label{didnt}
We expect that the category $\O$ admits a graded lift $\dot\O$
which is a $U_q \mathfrak{sp}_{2\infty}$-tensor product
categorification of $\dot V^{\otimes n}$. Then there should be a
$\Q(q)$-vector space isomorphism
$$
\dot \iota:\Q(q)\otimes_{\Z[q,q^{-1}]} K_0(\dot \O)
\stackrel{\sim}{\rightarrow}
\dot V^{\otimes n}, \quad
[\dot M(\ba)] \mapsto \dot v_\ba,
\quad
[\dot P(\bb)] \mapsto \dot o_\bb,
$$
for suitable graded lifts $\dot M(\ba)$ and $\dot P(\bb)$
of $M(\ba)$ and $P(\bb)$.
It should be possible to prove these statements by mimicking
the general approach developed in \cite{BLW}. The argument would also
yield an extension of the uniqueness theorem (Theorem~\ref{uniqueness}) from
$\mathfrak{sp}_{2k}$ to $\mathfrak{sp}_{2\infty}$.
\end{remark}

\subsection{Prinjectives and the associated crystal}\label{cryssec}
The proof of the uniqueness theorem in \cite{LW} gives a great
deal of additional information about the structure of TPCs of
$V_k^{\otimes n}$. In particular, \cite[Theorem 7.2]{LW} gives an
explicit combinatorial description of the associated 
crystal 
in the general sense of
\cite[$\S$4.4]{BD1}. Also,
\cite[Proposition 5.2]{LW} gives a classification of 
the indecomposable {\em prinjective} ($=$ projective and injective)
objects. Here is a precise statement of these results:

\begin{theorem}[Losev-Webster]\label{lw2}
Let $\C_k$ be a TPC of $V_k^{\otimes n}$.
Denote its distinguished irreducible objects by
$\{L_k(\bb)\:|\:\bb \in \B_k\}$.
\begin{enumerate}
\item
The associated crystal is the crystal structure on $\B_k$ defined in $\S$\ref{scr}.
This means that $F_i L_k(\bb) \neq 0$ (resp.\ $E_i L_k(\bb) \neq 0$) 
if and only if $\tilde f_i
\bb \neq \varnothing$ (resp. $\tilde e_i \bb \neq \varnothing$), in which case 
$F_i L_k(\bb)$ (resp. $E_i L_k(\bb)$)
has irreducible head and
socle isomorphic to $L_k(\tilde f_i\bb)$
(resp. $L_k(\tilde e_i \bb)$).
\item
The projective cover of $L_k(\bb)$ is injective if and
only if $\bb$ is antidominant, i.e. it is an element of the connected
component $\B_k^\circ$ of the crystal generated by the tuple $\bz_k$
from (\ref{zk}).
\end{enumerate}
\end{theorem}

Using also Theorem~\ref{fk} and letting $k \rightarrow \infty$, we get the
following corollary, which extends this result to infinite rank.

\begin{corollary}\label{thegraph}
Let $\C$ be a TPC of $V^{\otimes n}$ with irreducible objects
$\{L(\bb)\:|\:\bb \in \B\}$ (e.g., the category $\O$ from Theorem~\ref{typec}).
\begin{enumerate}
\item
The associated crystal is the crystal structure on $\B$ defined in $\S$\ref{scr}.
\item
The projective cover of $L(\bb)$ is injective if and
only if $\bb$ is antidominant.
\end{enumerate}
\end{corollary}

\begin{proof}
For (1), choose $k$ so that $i \in I_k$ and 
all of the composition factors of $F_i L(\bb)$ have label belonging to
$\B_k$.
Then, $F_i L(\bb) \in \ob \C_{\leq k}$, and its socle and head can be determined by
passing to the quotient category $\C_k$, where the result follows from Theorem~\ref{lw2}(1).
For (2), choose $k$ so that all composition factors of  the
projective cover of $L(\bb)$ have label belonging
to $\B_k$. Then we get done by Theorem~\ref{lw2}(2), since an object
of $\C$ with composition factors labelled by $\B_k$ is projective or
injective in $\C$ if and only if its image is projective or injective
in $\C_k$. 
\end{proof}

\section{Category $\mathcal F$}
To conclude the article, we formulate and prove a generalization of
\cite[Conjecture 5.13]{CKW}, then deduce some consequences for the
structure of the category $\mathcal F$ of finite-dimensional
half-integral weight $\g$-supermodules.
Throughout this section, $I$ denotes $\N$ and $I_0 := \Z_+$, 
i.e. they are the index sets for the simple roots of 
$\mathfrak{sp}_{2\infty}$ and $\mathfrak{sl}_{+\infty}$, respectively.

\subsection{\boldmath 
Truncation from $\mathfrak{sp}_{2\infty}$ to
$\mathfrak{sl}_{+\infty}$}\label{roger2}
Recall the $\mathfrak{sl}_{+\infty}$-module $V_0^{\otimes \bsigma}$
from $\S$\ref{tp}. We gave two different realizations of that, one as
a submodule of the $\mathfrak{sl}_\infty$-module $V^{\otimes
  \bsigma}$,
the other as a submodule of the $\mathfrak{sp}_{2\infty}$-module
$V^{\otimes n}$.
In turn, categorifications of $V_0^{\otimes \bsigma}$
can be constructed either by truncating from
a TPC of the $\mathfrak{sl}_\infty$-module $V^{\otimes\bsigma}$
as explained in \cite[$\S$2.8]{BLW}, or by truncating from a TPC of the
$\mathfrak{sp}_{2\infty}$-module $V^{\otimes n}$. In this subsection, we are going to
follow the latter route.

We begin with a couple more definitions. 
The {\em quiver Hecke category} $\QH_0$ of type
$\mathfrak{sl}_{+\infty}$ may be identified with the full subcategory
of the quiver Hecke category $\QH$
of type $\mathfrak{sp}_{2\infty}$ from Definition~\ref{QH1}
whose objects are monoidally generated by $I_0 \subset I$.

\begin{definition}\label{tpcdef2}
Fix $\bsigma = (\sigma_1,\dots,\sigma_n) \in \{\pm\}^n$.
A {\em TPC} of the $\mathfrak{sl}_{+\infty}$-module
$V_0^{\otimes \bsigma}$
is the following data:
\begin{itemize}
\item
a highest weight category $\mathcal C$ 
with 
standard objects $\{\Delta(\bb)\:|\:\bb \in \B_0\}$
indexed by the set $\B_0$ from (\ref{B0})
ordered according to the Bruhat order $\preceq_\bsigma$ from (\ref{o600});
\item
adjoint pairs $(F_i, E_i)$ of endofunctors of $\mathcal C$ for each $i
\in I_0$;
\item
a strict monoidal functor $\Phi:\QH_0 \rightarrow
\mathcal{E}nd(\C)$ with $\Phi(i) = F_i$ for each $i \in I_0$.
\end{itemize}
We impose the following additional axioms for all $i
\in I_0$ and $\bb \in \B_0$:
\begin{itemize}
\item
$E_i$ is isomorphic to a left adjoint
of $F_i$;
\item
$F_i \Delta(\bb)$ 
has a filtration with sections 
$\{\Delta(\bb+\sigma_t\bd_t)\:|\: 1 \leq t \leq n,
\isig_t^\bsigma(\bb) = \mathtt{f}\}$;
\item
$E_i \Delta(\bb)$ 
has a filtration with sections 
$\{\Delta(\bb-\sigma_t\bd_t)\:|\:1 \leq t \leq n, \isig_t^\bsigma(\bb) =
\mathtt{e}\}$;
\item
the natural transformation
$\Phi\Big(
\mathord{
\begin{tikzpicture}[baseline = -2]
	\draw[-,thick,darkred] (0.08,-.15) to (0.08,.3);
      \node at (0.08,0.05) {$\color{darkred}\bullet$};
   \node at (0.08,-.25) {$\scriptstyle{i}$};
\end{tikzpicture}
}
\Big)$ is {locally nilpotent}.
\end{itemize}
\end{definition}

View $\B$ as a
poset via the $\mathfrak{sp}_{2\infty}$-Bruhat order from
(\ref{o700}).
Recalling (\ref{N2}),
let $\B_{\leq \bsigma}$ be the set of all $\bb \in \B$ such that
$N_{[1,s]}(\bb, 0) \leq \sigma_1+\cdots+\sigma_s$ for $s=1,\dots, n-1$
and $N_{[1,n]}(\bb,0) = \sigma_1+\cdots+\sigma_n$.
Let $\B_{< \bsigma}$ be the set of all $\bb \in \B_{\leq \bsigma}$
such that $N_{[1,s]}(\bb,0) < \sigma_1+\cdots+\sigma_s$ for at least
one $s$.
Lemma~\ref{spequiv} implies that these are both ideals in $\B$.
Moreover, the set difference $\B_{\leq \bsigma} \setminus \B_{<
  \bsigma}$
is precisely the index set $\B_\bsigma$.

Now let $\C$ be a TPC of $V^{\otimes n}$ in the sense of
Definition~\ref{tpcdef}.
Let $\C_{\leq \bsigma}$ 
and $\C_{< \bsigma}$ be the Serre subcategories of $\C$
corresponding to the ideals $\B_{\leq \bsigma}$ and $\B_{< \bsigma}$,
respectively.
Then form the Serre quotient
$\C_\bsigma := \C_{\leq \bsigma} / \C_{< \bsigma}$.
This has a naturally induced structure of highest weight category with
weight poset $(\B_{\bsigma}, \preceq)$.
Its irreducibles $\{L_\bsigma(\bb)\:|\:\bb \in \B_\sigma\}$ 
are the canonical images of the $L(\bb)$'s.
The following parallels Theorem~\ref{fk}.

\begin{theorem}\label{jordan}
The subquotient $\C_\bsigma$ of $\C$ admits the structure of a
TPC of $V_0^{\otimes\bsigma}$.
\end{theorem}

\begin{proof}
Like in \cite[$\S$2.5]{BLW}, $\C_{\bsigma}$
is a highest weight category with weight poset $(\B_\sigma, \preceq)$, which is
isomorphic to $(\B_0, \preceq_{\bsigma})$ thanks to Lemma~\ref{threw}.
Also, the endofunctors $E_i, F_i$ for $i \in I_0$
leave both $\C_{\leq \bsigma}$ and $\C_{< \bsigma}$ invariant,
hence, they induce endofunctors of $\C_{\bsigma}$.
This follows by a similar argument to the proof of Theorem~\ref{fk};
the key point this time is that $\bb \in \B_0$ satisfies
$$
N_{[1,s]}(\bb\pm\bd_r,0) = N_{[1,s]}(\bb,0)
$$
whenever $\isig_r(\bb) \in \{\mathtt{e}, \mathtt{f}\}$
for some $i \in I_0$.
We should also note for $i \in I_0$, $\bb \in \B_\bsigma$, and
$\bb' \in \B_0$ defined via (\ref{restrict}) that:
\begin{itemize}
\item
$\isig_t(\bb) = \mathtt{e}$ (resp. $\mathtt{f}$) if and only if
$\isig^\bsigma_t(\bb') = \mathtt{e}$ (resp. $\mathtt{f}$);
\item $(\bb \pm \sigma_t \bd_t)' = \bb' \pm \bd_t$.
\end{itemize}
This follows from Lemma~\ref{restrict} using (\ref{haha}) and (\ref{cg}).
\end{proof}

The extremal choices for $\bsigma$ deserve some special mention.
For $\bsigma$ as in the following lemma, the subquotient $\C_\bsigma$
of Theorem~\ref{jordan}
may be identified with a subcategory of $\C$.

\begin{lemma}\label{ex1}
Suppose that
$\bsigma = (-,\dots,-,+,\dots,+)$ with 
$n_1$ entries equal to $-$ followed by $n_0$ entries equal to $+$.
Then 
$\B_\bsigma$ is an ideal in $\B$.
%, so that
%$\dot o_\bb = \pr_\bsigma \dot o_\bb$ and 
%$\dot c_\bb = \pr_\bsigma c_\bb$ for any $\bb \in \B_\bsigma$.
\end{lemma}

\begin{proof}
We actually show that $\B_{\bsigma} = \B_{\leq \bsigma}$, which is an ideal.
Take $\ba \in \B_{\leq \bsigma}$.
Since $N_{[1,n]}(\ba,0) = n_0-n_1$, exactly $n_1$ of the entries of $\ba$ are
$\leq 0$.
Since $N_{[1,n_1]}(\ba,0) \leq -n_1$, these must constitute the first $n_1$
entries of $\ba$. Hence, $\ba \in \B_\bsigma$.
\end{proof}

At the other extreme, for $\bsigma$ as in the next lemma, 
the subquotient $\C_\bsigma$ may be identified with a quotient of $\C$ itself.

\begin{lemma}\label{drive}
Suppose that
$\bsigma = (+,\dots,+,-,\dots,-)$ with 
$n_0$ entries equal to $+$ followed by $n_1$ entries equal to $-$.
Then $\B_\bsigma$ is a coideal (upper set) in $\B$.
\end{lemma}

\begin{proof}
We first observe that
\begin{equation}\label{bug}
\B_{\leq\bsigma} = \left\{\ba \in \B\:\big|\:
N_{[1,n]}(\ba,0) = n_0-n_1
\right\}.
\end{equation}
To see this,
any $\ba \in \B_{\leq \bsigma}$
satisfies $N_{[1,n]}(\ba,0) = \sigma_1+\cdots+\sigma_n = n_0-n_1$.
Conversely, if
$N_{[1,n]}(\ba,0) = n_0-n_1$, then
exactly $n_0$ of the entries of $\ba$ are
$> 0$ and $n_1$ entries are $\leq 0$.
Permuting the positive entries to the beginning makes the
numbers
$N_{[1,s]}(\ba,0)$ bigger, hence, $N_{[1,s]}(\ba,0) \leq
\sigma_1+\cdots+\sigma_s$ for all $s$. This shows $\ba \in
\B_{\leq\bsigma}$.

Now we can show that $\B_\bsigma$ is a coideal. Suppose that $\ba
\in \B_\bsigma$ and $\bb \succeq \ba$. Then $N_{[1,n]}(\bb,0) =
N_{[1,n]}(\ba,0)$, hence, $\bb \in \B_{\leq \bsigma}$. Since
$\B_\bsigma$ is a coideal in $\B_{\leq \bsigma}$, this implies that
$\bb \in \B_\bsigma$.
\end{proof}

\subsection{Proof of the second Cheng-Kwon-Wang conjecture}
TPCs of $V^{\otimes \bsigma}$ and $V_0^{\otimes \bsigma}$ are studied
in detail in \cite{BLW}. 
Combining results established there with Theorem~\ref{jordan}
and our main categorification theorem, recalling the
definition of the canonical and orthodox bases from
(\ref{pilly})--(\ref{polly}) and Theorem~\ref{lastyear},
we obtain the following:

\begin{theorem}\label{thethm}
Given $\bb \in \B$, define $\bsigma$ so that 
$\bb \in \B_\bsigma$, i.e. we take
$\sigma_r := +$ if
  $b_r > 0$ or $\sigma_r := -$ if $b_r \leq 0$.
Then, 
$\pr_\bsigma o_\bb = \pr_\bsigma c_\bb = \pr_0 c_{\bb'}^\bsigma$.
\end{theorem}

\begin{proof}
Remembering that $\O$ is a TPC of $V^{\otimes n}$ thanks to Theorem~\ref{typec},
let $\O_\bsigma := \O_{\leq \bsigma} / \O_{< \bsigma}$ be constructed
from $\O$ as in Theorem~\ref{jordan}.
For $\bb \in \B_\bsigma$,
the canonical image of $P(\bb)$ in the quotient category
$\O_\bsigma$
is the indecomposable projective object of this TPC of
$V_0^{\otimes\bsigma}$ indexed by $\bb'$.
By \cite[Corollary 5.30]{BLW}, its isomorphism class is identified
with $\pr_0
c_{\bb'} \in V_0^{\otimes\bsigma}$.
In view of the definition of $o_\bb$ from Theorem~\ref{lastyear}, this shows
that
$\pr_\bsigma o_\bb = \pr_0 c_{\bb'}$.
This 
equals $\pr_\bsigma c_\bb$
thanks to Lemma~\ref{ckwp}.
\end{proof}

\begin{corollary}\label{howto}
Suppose that $\ba, \bb \in \B$ have the property that $a_r > 0$ if and
only if $b_r > 0$ for each $r=1,\dots,n$.
Then, 
$(P(\bb):M(\ba)) = [M(\ba):L(\bb)] = d_{\ba,\bb}(1) = d^\bsigma_{\ba',\bb'}(1)$.
\end{corollary}

\begin{proof}
The first equality is BGG reciprocity in the highest weight category
$\O$.
Defining $\bsigma$ so that $\ba,\bb \in \B_\bsigma$, we can compute
$[M(\ba):L(\bb)]$
by passing to the quotient category $\O_\bsigma$ and computing the
corresponding composition multiplicity there. Theorem~\ref{thethm} tells us
that that is computed by the polynomials (\ref{pilly})--(\ref{polly})
evaluated at $q=1$.
\end{proof}

In particular,
if all of the strictly positive entries of $\bb \in \B$ appear after the
weakly negative ones, then Corollary~\ref{howto} plus Lemma~\ref{ex1} show
that all composition multiplicities in the Verma
supermodule $M(\bb)$ are determined by computing corresponding coefficients of 
canonical basis elements (either type A or C).
At the other extreme, using Lemma~\ref{drive} instead, if all of the strictly
positive entries of $\bb \in \B$ come before the weakly negative ones,
then the same is true for all of the Verma multiplicities in the projective
$P(\bb)$. We can state this formally in terms of the orthodox basis as follows:

\begin{corollary}\label{plick}
If $\bb \in \B$ has all its strictly positive entries appearing 
  before the weakly negative ones, then
$o_\bb = \pr_\bsigma o_\bb = \pr_\bsigma c_\bb = \pr_0
c^\bsigma_{\bb'}= c_\bb.$
\end{corollary}

This is exactly the situation of \cite[Conjecture 5.13]{CKW}, which 
follows easily from Corollary~\ref{plick} using also the Ringel duality of
\cite[(7.12)]{Btilt}. 

\begin{remark}
The $q$-analog of Theorem~\ref{thethm} is also true: in the setup of
the theorem,
we have that $\pr_\bsigma \dot o_\bb = \pr_\bsigma \dot c_\bb = \pr_0
\dot c_{\bb'}^\bsigma$.
If we had proved the assertions in Remark~\ref{didnt}, this would follow by
repeating the proof of Theorem~\ref{thethm} in the graded setting.
Without this, one needs a slightly more roundabout argument, involving
truncating
to 
$\mathfrak{sl}_k \hookrightarrow \mathfrak{sp}_{2k}$. 
Since we have not introduced notation for this, we omit the detailed
argument.
This implies also the $q$-analog of Corollary~\ref{plick}: we have that
$$
\dot o_\bb = \pr_\bsigma \dot o_\bb = \pr_\bsigma \dot c_\bb = \pr_0
\dot c_{\bb'}^\bsigma = \dot c_\bb
$$
in case all strictly positive entries of $\bb$ precede the weakly
negative ones.
\end{remark}

\subsection{Decomposition of category $\mathcal F$}
In this subsection, we view $\B$ as a poset via the
$\mathfrak{sp}_{2\infty}$-Bruhat order $\preceq$ from (\ref{o700}).
Given a decomposition $n = n_0+n_1$ with $n_0,n_1 \geq 0$, let
\begin{align}
\B_{n_0|n_1}&:= 
\left\{\bb \in \B\:|\:\text{$\bb$ has $n_0$ entries that are
    $> 0$ and $n_1$ entries that are $\leq 0$}\right\},\label{bb1}\\\label{bb2}
\B^\#_{n_0|n_1}&:= 
\left\{\bb \in \B\:|\:b_1,\dots,b_{n_0} > 0,
b_{n_0+1},\dots, b_n \leq 0\right\},\\
\B^+_{n_0|n_1}&:= 
\left\{\bb \in \B\:|\:b_1 > \cdots > b_{n_0} > 0 \geq b_{n_0+1} >
  \cdots > b_n\right\}.\label{bb3}
\end{align}

\begin{lemma}\label{crap}
Let $\bsigma = (+,\dots,+,-,\dots,-)$ with $n_0$ entries $+$ and $n_1$
entries $-$.
Then $\B_{n_0|n_1} =\B_{\leq \bsigma}$ and $\B_{n_0|n_1}^\# =
\B_\bsigma$.
In particular, 
$\B_{n_0|n_1}^\#$ is a coideal in $\B_{n_0|n_1}$.
\end{lemma}

\begin{proof}
The first equality follows from (\ref{bug}), and the second is clear
from (\ref{Bsigma}).
\end{proof}

\begin{lemma}\label{connection}
Any $\bb \in \B_{n_0|n_1}^+$
can be connected to a {\em typical} $\ba \in \B_{n_0|n_1}^+$ by
applying a
sequence of the crystal operators $\tilde e_i, \tilde
f_i\:(i \in I_0)$ from $\S$\ref{scr}.
\end{lemma}

\begin{proof}
We proceed by induction on the atypicality of $\bb \in
\B_{n_0|n_1}^+$,
i.e. the number of pairs $1 \leq r < s \leq n$ such that $b_r+b_s=1$.
If $\bb$ is typical, the result is trivial. So suppose that $\bb$ is
not typical. Let $r$ be minimal such that $b_r + b_s = 1$ for some $s
> r$. Set $i := b_r$, so that $b_s = 1-i$. Since $b_r > b_s$, we have
that $i > 0$. 
Then let $j \geq i$ be minimal such that
$\{j+1,-j\}\cap \{b_1,\dots,b_n\} = \varnothing$.

Now we make a second induction on $j-i$.
If $j = i$, then we let $\bc \in \B_{n_0|n_1}^+$ be obtained from
$\bb$ by replacing its
entry $i$ with $i+1$. Then $\bc$ is of smaller
atypicality than $\bb$.
Also $\bc = \tilde f_i \bb$, and we get
done by applying the first induction hypothesis to $\bc$.
If $j > i$, we either have that $j \in \{b_1,\dots,b_n\}$ or $1-j \in
\{b_1,\dots,b_n\}$, but not both (by the minimality of $r$).
In the former case, let $\bc \in \B_{n_0|n_1}^+$ be obtained from $\bb$ by replacing its 
entry $j$ with $j+1$; then, $\bc = \tilde f_j \bb$.
In the latter case, let $\bc \in \B_{n_0|n_1}^+$ be obtained from $\bb$ by replacing its
entry $1-j$ with $-j$; then, $\bc = \tilde e_j \bb$.
Either way, $\bc$ has the same atypicality as $\bb$, but the analog of
the statistic $j-i$ for $\bc$ is one less than it was for $\bb$.
It remains to apply the second induction hypothesis to $\bc$ to
finish the proof.
\end{proof}

Let
$\O_{n_0|n_1}$ be the Serre
subcategory of $\O$ generated by $\left\{L(\bb)\:\big|\:\bb \in
\B_{n_0|n_1}\right\}$.
One can determine whether $\bb \in \B$ belongs to $\B_{n_0|n_1}$ just
from knowledge of $|\WT(\bb)|$ (it does so if and only if $\sum_{i \in
  I} (|\WT(\bb)|,\eps_i) = n_0-n_1$). So Corollary~\ref{split}
implies that
$\O_{n_0|n_1}$ is a sum of blocks of
$\O$. Hence:
\begin{equation}\label{De}
\O = \bigoplus_{n_0+n_1=n} \O_{n_0|n_1}.
\end{equation}
Let $\mathcal F$ be the full subcategory of $\mathcal O$ consisting of
all finite-dimensional supermodules. 
Setting $\F_{n_0|n_1} := \F \cap \O_{n_0|n_1}$,
the decomposition (\ref{De}) induces a decomposition
\begin{equation}\label{haha}
\F = \bigoplus_{n_0+n_1=n} \F_{n_0|n_1}.
\end{equation}
By \cite[Theorem 4]{Pen},
the supermodule 
$L(\bb)$ is finite-dimensional if and only if
$\bb$ is {\em strictly dominant} in the sense that $b_1 > \cdots >
b_n$.
Consequently, $\mathcal F_{n_0|n_1}$ is the
Serre subcategory of $\mathcal O$ generated by 
$\big\{L(\bb)\:\big|\:\bb \in \B_{n_0|n_1}^+\big\}$.

The categorical $\mathfrak{sp}_{2\infty}$-action on $\O$ leaves the
subcategory $\F$ invariant; this follows
because the special projective
functors from (\ref{spc}) send finite-dimensional supermodules to
finite-dimensional supermodules. 
From this, we get induced categorical $\mathfrak{sl}_{+\infty}$-actions on
$\F_{n_0|n_1}\hookrightarrow\O_{n_0|n_1}$ for each $n_0+n_1=n$.
Recalling Lemma~\ref{crap}, let $\overline\O_{n_0|n_1}$ be the quotient of $\O_{n_0|n_1}$ by the
Serre subcategory generated by $\big\{L(\bb)\:\big|\:\bb \in \B_{n_0|n_1}
\setminus \B_{n_0|n_1}^\#\big\}$.
Writing $\overline{L}(\bb)$ for the canonical image of $L(\bb)$ in
$\overline\O_{n_0|n_1}$,
the irreducible objects of $\overline{\O}_{n_0|n_1}$ are represented by
$\big\{\overline{L}(\bb)\:\big|\:\bb \in \B_{n_0|n_1}^\#\big\}$.

\begin{lemma}\label{beans}
$\overline{\O}_{n_0|n_1}$ is a TPC of the $\mathfrak{sl}_{+\infty}$-module
$(V^+_0)^{\otimes n_0} \otimes (V_0^-)^{\otimes
  n_1}$. 
\end{lemma}

\begin{proof}
This follows from Lemma~\ref{crap} and Theorem~\ref{jordan},
since
$\overline{\O}_{n_0|n_1}$ is the same as the
quotient category $\O_{\bsigma}$ for $\bsigma$ as in that lemma.
\end{proof}

Let $\overline\F_{n_0|n_1}$ be the Serre subcategory of
$\overline\O_{n_0|n_1}$ generated
by $\big\{\overline{L}(\bb)\:\big|\:\bb \in \B^+_{n_0|n_1}\big\}$.
We are going to consider the following commutative diagram of functors:
\begin{equation}\label{roger}
\begin{CD}
\F_{n_0|n_1}&@>>>&\O_{n_0|n_1}\phantom{.}\\
@VQ VV&&@VVV\\
\overline{\F}_{n_0|n_1}&@>>>&\overline{\O}_{n_0|n_1}.
\end{CD}
\end{equation}
Here, the horizontal functors are the canonical inclusions, the
right hand functor is the quotient functor, and the commutativity of
the diagram then determines
the left hand functor $Q$ uniquely.
The categorical $\mathfrak{sl}_{+\infty}$-action on $\O_{n_0|n_1}$
induces an action on the quotient category
$\overline{\O}_{n_0|n_1}$. Then this restricts also to an action on
$\overline{\F}_{n_0|n_1}$.

\begin{lemma}\label{q}
The functor $Q:\mathcal F_{n_0|n_1} \rightarrow \overline{\mathcal
  F}_{n_0|n_1}$ is a strongly equivariant equivalence of
$\mathfrak{sl}_{+\infty}$-categorifications.
\end{lemma}

\begin{proof}
It is immediate from  the construction that $Q$ is strongly
equivariant. Also, since $\B^+_{n_0|n_1} \subseteq \B^\#_{n_0|n_1}$,
the images under $Q$ of all of the irreducible objects of $\mathcal
F_{n_0|n_1}$ are non-zero.
This is enough to show that $Q$ is fully faithful; cf. \cite[Lemma
2.13]{BD1}. It just remains to show that $Q$ is dense. 

As it is a Serre subcategory of the Schurian
category $\overline\O_{n_0|n_1}$,
the category $\overline\F_{n_0|n_1}$ is itself Schurian; in
particular, it has enough projectives.
For $\bb \in \B_{n_0|n_1}^+$, 
let $\overline{P}(\bb)$ be the projective cover of $\overline{L}(\bb)$
in $\overline\F_{n_0|n_1}$.
It suffices to show that each $\overline{P}(\bb)$ is a summand of
something in the essential
image of $Q$. Then, to get all other objects of
$\overline\F_{n_0|n_1}$, one can argue by considering a two-step projective
resolution, using the exactness of $Q$ and the Five Lemma.

Suppose in this paragraph that $\ba \in \B_{n_0|n_1}^+$ is typical.
Then the Verma supermodule $M(\ba)$ is projective in $\O_{n_0|n_1}$.
Hence, the projective object $\overline{P}(\ba)$ may be realized as the largest quotient of
the canonical image of $M(\ba)$ in $\overline{\O}_{n_0|n_1}$ which
belongs to $\overline{\F}_{n_0|n_1}$.
Typicality also implies that there are no strictly dominant $\bb \in \B$
with $\bb \prec \ba$. We deduce that this largest quotient is
$\overline{L}(\ba)$.
This shows that
$\overline{P}(\ba) = \overline{L}(\ba)$.

Now take any $\bb \in \B_{n_0|n_1}^+$.
Applying Lemma~\ref{connection}, we can find a typical $\ba \in \B_{n_0|n_1}^+$ 
connected to $\bb$ by a sequence of the crystal operators $\tilde e_i,
\tilde f_i\:(i \in I_0)$. In view of Corollary~\ref{thegraph}, it follows that
there is a sequence $X$ of the functors $E_i, F_i\:(i \in I_0)$ such
that $L(\bb)$ appears in the head of $X L(\ba)$.
Passing to the quotient category,
this shows that
$$
\Hom_{\overline{\F}_{n_0|n_1}}(X \overline{L}(\ba), \overline{L}(\bb)) \neq 0.
$$
By the previous paragraph, $\overline{L}(\ba)$ is projective in
$\overline{\F}_{n_0|n_1}$. Since $X$ has a biadjoint, it sends
projectives to projectives.
This means that
$X \overline{L}(\ba)$ is projective in $\overline{\F}_{n_0|n_1}$ too. We deduce that
$\overline{P}(\bb)$ is a summand of $X \overline{L}(\ba)$. 
Since $Q L(\ba) = \overline{L}(\ba)$ and $Q$ is strongly
equivariant, we have that
$Q (X L(\ba)) \cong X \overline{L}(\ba)$. Thus,
$\overline{P}(\bb)$ is a summand of something in the essential image of $Q$.
\end{proof}

\subsection{Realization of $\F_{n_0|n_1}$ via $\mathfrak{gl}_{n_0|n_1}(\CC)$}\label{damp}
Through the subsection, we fix $n_0,n_1 \geq 0$ with $n_0+n_1=n$.
The goal is to show that $\mathcal{F}_{n_0|n_1}$ is a highest
weight category. To do this, we are going to give
a different realization of the categories
$\overline{\mathcal{F}}_{n_0|n_1} \hookrightarrow
\overline{\mathcal{O}}_{n_0|n_1}$, 
then appeal to Lemma~\ref{q}.
We'll view $\B$ as a poset using the
$\mathfrak{sl}_\infty$-Bruhat order $\preceq_\bsigma$ from (\ref{o600}),
taking $\bsigma := (+,\dots,+,-,\dots,-)$ with $n_0$ entries $+$ and
$n_1$ entries $-$.
Recall also the subset $\B_0$ of $\B$ from (\ref{B0}).
Let
\begin{align}\label{train}
\B^{n_0|n_1} &:= \{\bb \in \B\:|\:b_1 > \cdots > b_{n_0}, b_{n_0+1} < \cdots
< b_{n}\},&
\B_0^{n_0|n_1} &:= \B^{n_0|n_1} \cap \B_0. 
\end{align}
Recalling the posets (\ref{bb2})--(\ref{bb3}) from the previous subsection,
the map $\bb \mapsto \bb'$ from (\ref{primemap}) defines poset isomorphisms
$\B^\#_{n_0|n_1}\stackrel{\sim}{\rightarrow}
\B_0$ and $\B_{n_0|n_1}^+ \stackrel{\sim}{\rightarrow} \B_0^{n_0|n_1}$.

\begin{lemma}\label{rafa}
The subsets $\B_0$ and $\B_0^{n_0|n_1}$ are coideals in $\B$ and
$\B^{n_0|n_1}$, respectively.
\end{lemma}

\begin{proof}
This follows from Lemma~\ref{slequiv}, on noting that
$$
\B_0 = \left\{\bb \in \B\:\Big|\:N_{[1,n_0]}^{\bsigma}(\bb,0)
\geq n_0, 
N_{[1,n]}^\bsigma(\bb,0) = n_0-n_1\right\},
$$
where $\bsigma = (+,\dots,+,-,\dots,-)$ as usual.
\end{proof}

Now we consider the general linear Lie superalgebra
$\mathfrak{g}':=
\mathfrak{gl}_{n_0|n_1}(\CC)$.
Let $\mathfrak{h}'$ and $\mathfrak{b}'$ be the Cartan subalgebra and
Borel subalgebra of $\mathfrak{g}'$ consisting of diagonal and upper
triangular matrices, respectively.
Let $\delta_1',\dots,\delta_{n}'$ be the basis for
$(\mathfrak{h}')^*$
dual to the diagonal matrix units in $\mathfrak{h}'$. Then define
$\O_{n_0|n_1}'$ to be the category of all $\mathfrak{g}'$-supermodules $M$ such
that
\begin{itemize}
\item $M$ is finitely generated over $\g'$;
\item $M$ is locally finite-dimensional over $\b'$;
\item $M$ is semisimple over $\h'$ with all weights of the form
$\lambda'_\bb$ for $\b \in \B$, where
$$
\lambda'_\bb := \sum_{r=1}^{n}
\lambda'_{\bb,r} \delta_r'
\qquad\text{where}\qquad
\lambda'_{\bb,r} = \left\{\begin{array}{ll}
b_r + r-1&\text{if $1 \leq r \leq n_0$,}\\
-b_r+r-2n_0&\text{if $n_0+1 \leq r \leq n$;}\\
\end{array}\right.
$$
\item
for $\bb \in \B$,
the $\Z/2$-grading on the $\lambda_\bb'$-weight space of $M$ 
is concentrated in parity $\sum_{r=n_0+1}^{n} \lambda_{\bb,r}'\pmod{2}$.
\end{itemize}
Note that $\O_{n_0|n_1}'$ is exactly the same as the Abelian category $\mathscr O$
defined in \cite[Lemma 2.2]{Bo}. 
It is a special case of the category
constructed in \cite[Definition 3.7]{BLW}, taking the type
$(\underline{n},\underline{c})$ there
to be $((1^n), (0^{n_0}, 1^{n_1}))$.
In particular, \cite[Theorem 3.10]{BLW} verifies the following:

\begin{lemma}\label{venus}
The category $\O_{n_0|n_1}'$ admits additional structure making it into a TPC
of the 
$\mathfrak{sl}_\infty$-module $(V^+)^{\otimes n_0}
\otimes (V^-)^{\otimes n_1}$.
\end{lemma}

Let us give a little more detail about the highest weight structure here.
The irreducible objects of $\O_{n_0|n_1}'$ are parametrized naturally by their
highest weights. We denote the one of highest weight $\lambda_{\bb}'$ by
$L'(\bb)$. It can be constructed explicitly as the unique irreducible
quotient of the corresponding 
Verma supermodule $M'(\bb)$. This is the standard object in the
highest weight category $\O_{n_0|n_1}'$ indexed by $\bb \in \B$.

Next, let $\F_{n_0|n_1}'$ be the subcategory of $\O_{n_0|n_1}'$ consisting of all of the
finite-dimensional supermodules. Note $\F_{n_0|n_1}'$ may also be described as
the Serre subcategory of $\O_{n_0|n_1}'$ generated by the irreducible
objects $\big\{L'(\bb)\:\big|\:\bb \in \B^{n_0|n_1}\big\}$.
This follows from Kac' classification of finite dimensional
$\mathfrak{g}'$-supermodules in \cite{Kac}. The argument there
realizes each of the finite-dimensional $L'(\bb)$ as a quotient of a
corresponding {\em Kac supermodule} $K'(\bb)$. 

The categorical $\mathfrak{sl}_\infty$-action on $\O_{n_0|n_1}'$ restricts to an
action on $\F_{n_0|n_1}'$. Taking the type $(\underline{n},
\underline{c})$ of \cite[Definition 3.7]{BLW} to be
$((n_0, n_1), (0,1))$, we get the following 
as another special case of \cite[Theorem 3.10]{BLW}, recalling also
\cite[Definition 2.10]{BLW} for this more general sort of TPC.

\begin{lemma}\label{serena}
The category $\F_{n_0|n_1}'$ is a TPC
of the 
$\mathfrak{sl}_\infty$-module $\bigwedge^{n_0} V^+ \otimes
\bigwedge^{n_1} V^-$.
\end{lemma}

Part of the content of Lemma~\ref{serena} is that $\F_{n_0|n_1}'$ is a highest
weight category. Its
standard objects are the Kac supermodules $\big\{K'(\bb)\:\big|\:\bb\in\B^+\big\}$ mentioned already.

The idea now is to truncate $\F_{n_0|n_1}' \hookrightarrow \O_{n_0|n_1}'$ from
$\mathfrak{sl}_\infty$ to $\mathfrak{sl}_{+\infty}$ to obtain our
alternate realization of the categories $\overline{\F}_{n_0|n_1}
\hookrightarrow \overline{\O}_{n_0|n_1}$. The construction we need for
this has already been developed in \cite[$\S$2.8]{BLW} (and is
entirely analogous to $\S\S$\ref{roger1}--\ref{roger2} above).

Recalling Lemma~\ref{rafa}, let $\overline{\O}_{n_0|n_1}'$ be the quotient of $\O_{n_0|n_1}'$
by the Serre subcategory generated by $\big\{L'(\bb)\:\big|\:\bb \in \B
\setminus \B_0\big\}$.
Denoting the canonical image of $L'(\bb)$ in $\overline{\O}_{n_0|n_1}'$ by
$\overline{L}'(\bb)$, the irreducible objects of $\overline{\O}_{n_0|n_1}'$ are
represented by $\big\{\overline{L}'(\bb)\:\big|\:\bb \in \B_0\big\}$.
Let $\overline{\F}_{n_0|n_1}'$ be the Serre subcategory of $\overline{\O}_{n_0|n_1}'$
generated by $\big\{\overline{L}'(\bb)\:\big|\:\bb \in \B_0^+\big\}$.
Analogously to (\ref{roger}), we get a commutative diagram of functors:
\begin{equation}
\begin{CD}
\F_{n_0|n_1}'&@>>>&\O_{n_0|n_1}'\phantom{.}\\
@VR VV&&@VVV\\
\overline{\F}_{n_0|n_1}'&@>>>&\overline{\O}_{n_0|n_1}'.
\end{CD}
\end{equation}
The categorical $\mathfrak{sl}_{\infty}$-actions on $\O'_{n_0|n_1}$ and
$\F'_{n_0|n_1}$
restrict to actions of $\mathfrak{sl}_{+\infty}$. These then induce
categorical $\mathfrak{sl}_{+\infty}$-actions on 
$\overline{\O}'_{n_0|n_1}$ and $\overline{\F}_{n_0|n_1}'$, so that all
of the above functors are strongly equivariant.

\begin{lemma}\label{last}
Let $\widetilde{\F}_{n_0|n_1}'$ be the quotient of $\F_{n_0|n_1}'$ by
the Serre subcategory generated by
$\big\{L'(\bb)\:\big|\:\bb \in \B^{n_0|n_1} \setminus
\B^{n_0|n_1}_0\big\}$.
The functor $R:\F_{n_0|n_1}' \rightarrow \overline{\F}_{n_0|n_1}'$
induces an
equivalence
$\widetilde{R}:\widetilde{\F}_{n_0|n_1}' \rightarrow \overline{\F}_{n_0|n_1}'$.
\end{lemma}

\begin{proof}
By the universal property of Serre quotients, $R$ induces
$\widetilde{R}: \widetilde{\F}_{n_0|n_1}'\rightarrow \overline{\F}_{n_0|n_1}'$. As in the proof of Lemma~\ref{q}, $\widetilde{R}$ is
fully faithful. To show that it is dense, we show equivalently
that $R$ is dense, again by mimicking the arguments from the proof of
Lemma~\ref{q}.
This involves replacing the notion of atypicality and the crystal
structure used in the proof of that lemma with their counterparts in
the category $\O'$. For $\bb \in \B^{n_0|n_1}$, its atypicality is the
number of pairs $1 \leq r < s \leq n$ such that $b_r = b_s$.
The appropriate crystal structure, and the required analog of
Corollary~\ref{thegraph}, are described in \cite[Lemma 2.23]{BLW}.
Actually, the bijection $\B^+_{n_0|n_1}
\stackrel{\sim}{\rightarrow}
\B^{n_0|n_1}_0, \bb \mapsto \bb'$ preserves atypicality, and intertwines the crystal operators
$\tilde e_i, \tilde f_i\:(i \in I_0)$ from $\S$\ref{connection} with the crystal operators
$\tilde e_i, \tilde f_i\:(i \in I_0)$ defined in \cite{BLW}. Then
the argument in the proof of Lemma~\ref{q} (dependent especially on
the combinatorial
Lemma~\ref{connection}) carries over almost immediately.
\end{proof}

\begin{lemma}\label{most}
The categories $\overline{\O}_{n_0|n_1}'$ and $\overline{\F}_{n_0|n_1}'$ are TPCs of $(V_0^+)^{\otimes n_0} \otimes
(V_0^-)^{\otimes n_1}$ and $\bigwedge^{n_0} V_0^+ \otimes
\bigwedge^{n_1} V_0^-$, respectively.
\end{lemma}

\begin{proof}
For $\overline{\O}_{n_0|n_1}'$, our statement follows immediately as a
special case of \cite[Theorem 2.19]{BLW}.
The same result shows that $\widetilde{\F}_{n_0|n_1}'$ is a TPC of
$\bigwedge^{n_0} V_0^+ \otimes
\bigwedge^{n_1} V_0^-$. It remains to appeal to Lemma~\ref{last} to get
the result for $\overline{\F}_{n_0|n_1}'$.
\end{proof}

\begin{theorem}\label{tax}
The $\mathfrak{sl}_{+\infty}$-categorifications 
$\overline{\mathcal O}_{n_0|n_1}$ and $\overline{\mathcal O}'_{n_0|n_1}$ are
strongly equivariantly equivalent via an equivalence which sends
$\overline{L}(\bb)$
to a copy of $\overline{L}'(\bb')$ for each $\bb \in \B_{n_0|n_1}^\#$.
\end{theorem}

\begin{proof}
In Lemmas~\ref{beans} and \ref{most}, we have shown that both categories are TPCs of $(V_0^+)^{\otimes n_0}
\otimes (V_0^-)^{\otimes n_1}$.
Now the result follows from the uniqueness theorem for such TPCs, which
is a special case of \cite[Theorem 2.12]{BLW}.
\end{proof}

\begin{corollary}\label{much}
The $\mathfrak{sl}_{+\infty}$-categorifications 
$\overline{\F}_{n_0|n_1}$ and $\overline{\mathcal F}'_{n_0|n_1}$ are
strongly equivariantly equivalent via an equivalence which sends
$\overline{L}(\bb)$
to a copy of $\overline{L}'(\bb')$ for each $\bb \in \B_{n_0|n_1}^+$.
\end{corollary}

\begin{proof}
Recall $\overline{\F}_{n_0|n_1}$ is the Serre subcategory of
$\overline{\O}_{n_0|n_1}$ generated by $\big\{\overline{L}(\bb)\:\big|\:\bb
\in \B_{n_0|n_1}^+\big\}$,
$\overline{\F}'_{n_0|n_1}$ is the Serre subcategory of
$\overline{\O}'_{n_0|n_1}$ generated by
$\big\{\overline{L}'(\bb)\:\big|\:\bb \in \B^{n_0|n_1}_0\big\}$,
and the map $\bb \mapsto \bb'$ is a bijection between $\B_{n_0|n_1}^+$
and $\B^{n_0|n_1}_0$.
Then apply Theorem~\ref{tax}.
\end{proof}

\begin{corollary}
The category $\mathcal F_{n_0|n_1}$ is a TPC of $\bigwedge^{n_0} V_0^+
\otimes \bigwedge^{n_1} V_0^-$. In particular, it is a highest weight
category with 
weight poset $(\B_{n_0|n_1}^+, \preceq)$ and irreducible objects
represented by $\big\{L(\bb)\:\big|\:\bb \in \B_{n_0|n_1}^+\big\}$.
\end{corollary}

\begin{proof}
This follows from Lemma~\ref{q}, Corollary~\ref{much} and Lemma~\ref{most}.
\end{proof}

\subsection{Realization of $\F$ via arc algebras}
In the final subsection, we are going to briefly explain another realization of the
category $\F$ in terms of the generalized Khovanov arc
algebras of \cite{BS1}. 
We will assume the reader is familiar with the language and constructions in
\cite{BS1, BS4}.

Let $\Lambda$ be the set of weights in the diagrammatic sense of
\cite[$\S$2]{BS1}
drawn on a number line with vertex set $I_0$, such that 
the number of vertices labelled $\times$ plus the number of vertices
labelled $\circ$ plus two times the number of vertices labelled
$\scriptstyle\vee$ is equal to $n$;
all of the (infinitely many) remaining vertices are labelled $\scriptstyle\wedge$.
The set $\Lambda$ is in bijection with
\begin{equation}\label{lunch}
\B^+ := \{\bb \in \B\:|\:b_1>\cdots>b_n\} = \bigcup_{n_0+n_1=n}
\B^+_{n_0|n_1}
\end{equation}
according to the following {\em weight dictionary}.
Given $\bb \in \B^+$, let 
\begin{align*}
I_\down(\bb) &:= \{b_r\:|\:r=1,\dots,n, b_r > 0\}\\
I_\up(\bb) &:= I_0 \setminus \{1-b_r\:|\:r=1,\dots,n, b_r \leq 0\}.
\end{align*}
Then we identify $\bb$ with the element of $\Lambda$ whose $i$th
vertex is labelled 
\begin{equation}\label{wtdict}
\left\{
\begin{array}{ll}
\circ&\text{if $i$ does not belong to either $I_\down(\la)$ or $I_\up(\la)$,}\\
{\scriptstyle\down} &\text{if $i$ belongs to $I_\down(\la)$ but not to
$I_\up(\la)$,}\\
{\scriptstyle\up} &\text{if $i$ belongs to $I_\up(\la)$ but not to
$I_\down(\la)$,}\\
{\scriptstyle\times}&\text{if $i$ belongs to both $I_\down(\la)$ and $I_\up(\la)$.}
\end{array}\right.
\end{equation}
Let $K_\Lambda$ be the generalized Khovanov algebra associated to the
set $\Lambda$ as defined in \cite{BS1}. This is a basic algebra
with isomorphism classes of irreducible representations indexed in a
canonical way by the set $\Lambda$.

\begin{theorem}\label{arcy}
There is an equivalence of categories between $\F$ and the
category $K_\Lambda\text{-}\mod$ of finite-dimensional left
$K_\Lambda$-modules.
It sends $L(\bb)\:(\bb\in\B^+)$ to the irreducible $K_\Lambda$-module indexed by the
element of $\Lambda$ associated to $\bb$ according to the above weight dictionary.
\end{theorem}

\begin{proof}
Corresponding to the decomposition (\ref{lunch}), we have that
$\Lambda = \bigcup_{n_0+n_1=n} \Lambda(n_0|n_1)$ where
$\Lambda(n_0|n_1)$ consists of the weights in $\Lambda$ whose diagrams
have $n_0$ entries equal to
$\scriptstyle\down$ or $\times$ and $n_1$ entries equal to
$\scriptstyle \down$ or $\circ$.
The algebra $K_\Lambda$ decomposes as $\bigoplus_{n_0+n_1=n}
K_{\Lambda(n_0|n_1)}$.
In view of (\ref{haha}), to prove the theorem, it suffices to show that $\F_{n_0|n_1}$ is
equivalent to $K_{\Lambda(n_0|n_1)}\text{-}\mod$.

By Lemma~\ref{q}, Lemma~\ref{last} and Corollary~\ref{much},
$\F_{n_0|n_1}$ is equivalent to the quotient
$\widetilde{\F}_{n_0|n_1}'$
of $\F_{n_0|n_1}'$ by the Serre
subcategory generated by $\big\{L'(\bb)\:\big|\:\bb \in
\B^{n_0|n_1}\setminus \B_0^{n_0|n_1}\}$.

By the main theorem of \cite{BS4}, $\F_{n_0|n_1}'$ is equivalent to
the category $K_{\Delta}\text{-}\mod$ of finite-dimensional
modules over another arc algebra $K_{\Delta}$. The set $\Delta$ of
weights this time are drawn on a number line with vertex set $\Z$,
such that the number of vertices labelled $\scriptstyle\down$ or
$\times$ is $n_0$, and the number labelled $\scriptstyle\down$ or
$\circ$ is $n_1$. Under the weight dictionary from the introduction of
\cite{BS4},
the set $\B_0^{n_0|n_1}$ is identified with the subset $\Delta_0$ of $\Delta$ 
consisting of weights $\bb$ whose diagrams have label $\scriptstyle \up$
on vertex $i$ for all $i \leq 0$.

We conclude that $\widetilde{\F}_{n_0|n_1}'$ is equivalent to the
category of finite-dimensional modules over the algebra
$\bigoplus_{\ba,\bb \in \Delta_0} e_\ba K_\Delta e_\bb$, where $e_\bb$
denotes the primitive idempotent in $K_\Delta$ indexed by $\bb$. 
Noting that $\Delta_0$ is in bijection with
$\Lambda(n_0|n_1)$ via the map which deletes all vertices indexed by
$\Z_{\leq 0}$,
this algebra is obviously isomorphic to $K_{\Lambda(n_0|n_1)}$.
\end{proof}

Theorem~\ref{arcy} has a number of consequences for the structure of the
category $\F$. We refer to the introduction of \cite{BS4} for a
comprehensive list: the present situation is entirely analogous.
It shows moreover that any block of $\F$ of atypicality $r$ (which in
the diagrammatic setting is the number of vertices labelled
$\scriptstyle\vee$ in weights belonging to the block) 
is Morita equivalent
to the algebra $K^{+\infty}_r$ from \cite{BS1}.
Thus, the category $\F$ gives the first known occurrence ``in nature'' of the
algebras $K^{+\infty}_r$.

\end{document}